\documentclass{amsart}
\usepackage{amsmath}
\usepackage{amssymb}
\usepackage{hyperref}
\usepackage[dvips]{graphicx}
\usepackage{epsfig}
\usepackage{psfrag}

\usepackage{pb-diagram}
\usepackage{appendix}
\usepackage{amsthm}

\usepackage{appendix}
\usepackage{color}
\definecolor{darkgreen}{cmyk}{1,0,1,.2}
\definecolor{m}{rgb}{1,0.1,1}
\definecolor{green}{cmyk}{1,0,1,0}

\definecolor{test}{rgb}{1,0,0}
\definecolor{cmyk}{cmyk}{0,1,1,0}

\newcommand{\leftsub}[2]{{\vphantom{#2}}_{#1}{#2}}

\newtheorem{Equation}{}[section]
\newtheorem{assumption}[Equation]{Assumption}

\newtheorem{theorem}[Equation]{Theorem}
\newtheorem{proposition}[Equation]{Proposition}
\newtheorem{lemma}[Equation]{Lemma}
\newtheorem{corollary}[Equation]{Corollary}

\newtheorem{definition}[Equation]{Definition}

\newtheorem{remark}[Equation]{Remark}

\def\Dom{\operatorname{Dom}}

\def\Im{\operatorname{Im}}

\def\Ker{\operatorname{Ker}}

\def\maE{\mathcal{E}}

\def\maK{\mathcal{K}}

\def\maL{\mathcal{L}}

\def\maG{\mathcal{G}}
\def\maF{\mathcal{F}}

\def\C{\mathbb C}

\def\R{\mathbb R}

\def\N{\mathbb N}

\def\maE{{\mathcal E}}
\def\maF{{\mathcal F}}

\def\maG{{\mathcal G}}
\def\cG{{\mathcal G}}

\def\td{{\widetilde d}}

\def\tU{\widetilde{U}}

\def\ep{\epsilon}

\begin{document}

\title[HP-complexes and leafwise maps]{Leafwise homotopies and Hilbert-Poincar\'{e} complexes\\
I. Regular HP-complexes and leafwise pull-back maps\\ \today}

\author[M-T. Benameur]{Moulay-Tahar Benameur}
\address{UMR 7122 du CNRS, Universit\'{e} Paul Verlaine-Metz, Ile du Saulcy, Metz, France}
\email{benameur@math.univ-metz.fr}

\author[I. Roy]{Indrava Roy}
\address{UMR 7122 du CNRS, Universit\'{e} Paul Verlaine-Metz, France\\and
Mathematics department, Institut f\"{u}r Mathematik, Universit\"{a}t Paderborn, Germany}
\email{indrava@gmail.com}

\date{1 May 2011}

\maketitle
\tableofcontents

\section{Introduction}

This paper is a first of a series of three papers which study some secondary homotopy invariants for laminations. More precisely, we build up a suitable
 framework
for the study of leafwise signature invariants which allows to deduce important consequences for the leafwise  homotopy classification
of laminations.
 In the third paper of this series, and under a usual Baum-Connes assumption, the second author deduces for instance that the type II Cheeger-Gromov rho invariant
associated with the leafwise  signature operator on an odd dimensional lamination, which was introduced in \cite{BenameurPiazza}, is a leafwise oriented
 homotopy invariant.
This result generalizes the case of closed odd dimensional manifolds \cite{KeswaniI, Weinberger, CW, Chang, Mathai, PiazzaSchick} and also the partial
results for
foliated topological bundles obtained in \cite{BenameurPiazza}, and we believe that the techniques involved are deep enough to enjoy their own interest.

In their work on mapping surgery to analysis, Nigel Higson and John Roe have sytematically studied the so-called Hilbert-Poincar\'{e} (abbreviated HP)
complexes \cite{MischenkoFomenko} and deduced interesting topological consequences (cf. \cite{HigsonRoe}). They defined an HP complex as a complex of
finitely-generated projective Hilbert $C^*$-modules on a $C^*$-algebra $A$ with adjointable differentials, and an additional structure of adjointable Poincar\'{e}
duality operators that induce isomorphism on cohomology from the original complex to its dual complex. Associated with an HP-complex there is a canonically
defined class in $K_1(A)$, called the signature of the HP-complex. It is shown in \cite{HigsonRoe} that a homotopy equivalence of such complexes leaves the
signature class invariant and yields an explicit homotopy, a path between the correspondng representatives of the signature classes. This explicit path
was used
in \cite{Keswani} to construct a controlled path of operators joining the signature operators on homotopy equivalent manifolds. When trying to
extend the results
of Keswani to general laminations, and in fact already to general smooth foliations, we faced the following difficulties:

\begin{enumerate}
\item  It is necessary to work with complete transversals and with the (maximal) $C^*$-algebras associated with the monodromy groupoids associated with
 the
 transversals. Therefore the HP-complexes that naturally arose were not defined over isomorphic $C^*$-algebras but only Morita equivalent ones.
\item For Galois coverings which correspond to a lamination with one leaf, the Whitney isomorphism allows to reduce the study of the de Rham complex to
a finitely generated projective
HP-complex as introduced in \cite{HigsonRoe}. Already for general foliations on closed manifolds, this reduction becomes highly involved and it is important
 thus to
extend the Higson-Roe formalism to countably
generated  HP-complex with regular operators.
\item Since an oriented leafwise homotopy equivalence induces a Morita equivalence of the $C^*$-algebras, the explicit path joining the leafwise signature
operators on the equivalent laminations needs to be rethought  up to an explicit imprimitivity bimodule.
\item The construction of the Keswani loop of unitaries  is done  by concatenating three paths. The first one is the
``spectral flow'' path which can be defined easily and whose Fuglede-Kadison log-determinant has to be related with  the measured
Cheeger-Gromov rho
invariant \cite{BenameurPiazza}. The second path uses the above mentionned path of operators associated with the leafwise homotopy and
is called the Large Time Path (LTP) while a third path called the Small Time Path (STP) needs a suitable
 description of the Baum-Connes map for laminations. This latter use of the Baum-Connes map turns out to be the hardest part of this work.
\end{enumerate}

All these problems  are solved in this paper, the second paper \cite{BenameurRoyII} by the second author and the third paper in prepration also by
the second author \cite{RoyIII}.
In this first paper, we begin by  extending the results of \cite{HigsonRoe} to deal with regular (unbounded) operators and more importantly  to take
into account Morita equivalence of underlying $C^*$-algebras. Given an oriented leafwise map, satisfying some natural assumption fulfilled by leafwise homotopies, between leafwise  oriented laminations on compact spaces in the sense
of \cite{MooreSchochet},  we  construct a pull-back morphism between the leafwise de Rham HP complexes and prove the expected (up to Morita equivalence)
functoriality. When two leafwise oriented laminations are leafwise homotopy equivalent, our construction allows to deduce an explicit path
joining the leafwise
signature operators and hence the LTP path for laminations. So, this first paper does not use any measure theory and is
only concerned with the $C^*$-algebraic
constructions associated with leafwise homotopies.
In the second paper \cite{BenameurRoyII} of this series, holonomy invariant transverse measures are introduced and a semi-finite von Neumann algebra
associated with the leafwise
homotopy equivalence between the laminations is constructed. Moreover, the second author  shows there that the measured laminated Cheeger-Gromov rho invariant
 \cite{BenameurPiazza} associated with  the leafwise signature operator is a foliated diffeomorphism invariant which does not depend on the leafwise
metric used to define it, extending the result of Cheeger-Gromov \cite{CheegerGromov}. Moreover, using the above mentioned von Neumann algebra with
its trace, he proves that the measured Fuglede-Kadison determinant
of the LTP cancels out in the large time limit. The third paper \cite{RoyIII} exploits a nice description of the Baum-Connes map to construct the STP which, when concatenated with the two previous paths
yields the allowed loop of unitaries. The Fuglede-Kadison determinant of the STP defined using again the above von Neumann algebra and its trace, is also proved to
cancel out in the small time limit. So combining the results of this series, we obtain the leafwise homotopy invariance of the measured rho invariant of the lamination.\\

Let us now briefly explain the main results of this first paper. \\

With a lamination $(M,\mathcal{F})$ on a compact space $M$, and a chosen complete transversal $X$ in the sense of \cite{MooreSchochet},
we associate the HP-complex
$$
\mathcal{E}^0_X\xrightarrow{d_X} \mathcal{E}^1_X\xrightarrow{d_X} \cdots \xrightarrow{d_X} \mathcal{E}^p_X
$$
where $\mathcal{E}^k_X$ is the completion of $C_c^{\infty, 0}(\mathcal{G}_X,r^*(\Lambda^kT^*\mathcal{F}))$
with respect to a $C^*(\mathcal{G}^{X}_{X})$-valued inner product (see Section \ref{HP} for the notations). The Poincar\'{e} duality operator,
denoted $T_X$, is  induced on $\mathcal{E}^k_X$ by the lift of the leafwise Hodge $*$-operator on $\mathcal{G}_X$, and $d_X$ is the regular operator
 induced by the lift of the leafwise de Rham differential to $\mathcal{G}_X$. Now consider a leafwise homotopy equivalence between two laminations
 $f:(M,\mathcal{F})\rightarrow (M',\mathcal{F}')$ and let $X'$ be similarly a complete transversal on $(M',\mathcal{F}')$. Then  we use the explicit
imprimitivity bimodule which implements the Morita equivalence to reduce to leafwise de Rham  HP-complexes over the same $C^*$-algebra
$C^*(\maG^{X'}_{X'})$.
Consequently, we define out of a leafwise homotopy equivalence a homotopy equivalence between the leafwise de Rham HP-complexes of the two laminations.
As a consequence, we deduce
the well-known equality of the $K_1$ signature classes of the two foliations, and more importantly an explicit path whose determinant will
play a fundamental part
in
\cite{BenameurRoyII}  as explained above.\\

The contents of the present paper are as follows. Section 2 is devoted to the generalization of the definitions and properties of HP-complexes to the
setting of regular (unbounded) operators. In Section 3, we construct the pull-back chain map, associated with an oriented leafwise map satisfying
a convenient condition fulfilled by leafwise homotopies,
between the corresponding leafwise de Rham HP-complexes, and prove its functoriality. Section 4 deals with oriented leafwise homotopy equivalences
and proves that such an equivalence induces a homotopy between the HP-complexes and an explicit path between the leafwise signature operators.\\

The main difficulties appear in the transition  from smooth manifolds to smooth foliations. The extension of our results
from smooth foliations to leafwise smooth
laminations, although highly important for applications, is  almost cosmetic. Therefore we have eventually decided  to avoid unnecessary heavy notations which would mislead the reader
and we have restricted to smooth foliations.  \\

{\em Acknowledgements.}
This work arose from discussions  with  P.  Piazza and we would like to thank him for his helpful advices during the preparation of this work.
We are also  indebted to B. Bekka, A. Gorokhovsky, J. Heitsch, N. Higson, J. Hilgert, V. Mathai, J. Roe and T. Schick for several
helpful discussions. Part of this work was done while the second author was visiting the department of mathematics
of the university La Sapienza in Rome, and while the first author was visiting the IHP in Paris. Both authors are grateful to the two institutions
 for their support and for the warm hospitality.

\section{Review of Hilbert-Poincar\'e complexes}\label{HP}

We review in this section some basic properties of a so-called Hilbert-Poincar\'e complex and collect some results that will be used in the sequel. We mainly extend some results proved in \cite{HigsonRoe} to encompass closed operators. More precisely, we put forward the necessary information required to define the signature of an HP-complex, the notion of homotopy equivalence of two HP-complexes and a crucial theorem, originally due to Higson and Roe and adapted here to our context, which states that homotopy equivalent HP-complexes have the same signature. We refer to \cite{Lance, Kustermans, Pal} for the detailed properties of Hilbert modules and regular operators that will be used here.

For a $C^*$-algebra $A$ and  Hilbert right $A$-modules $\maE$ and $\maE'$, we denote by $\maL_A (\maE, \maE')$ the set of adjointable operators from $\maE$ to $\maE'$, i.e. of maps $T:\maE \to \maE$ such that there exists a map $S:\maE \to \maE$ with the property
$$
\text{ For $v\in \maE$ and $v'\in \maE'$, }< T (v), v'> = < v, S(v')> \quad  \in A.
$$
It is then easy to see that such $S$ is unique, $A$-linear and bounded, and it will be called the adjoint of $T$ and denoted $T^*$ as usual. A trivial consequence of the definition of  an adjointable operator is that it is automatically $A$-linear and bounded. Moreover, $\maL_A(\maE, \maE')$ is naturally endowed with the structure of a Banach space. When $\maE'=\maE$, we simply denote $\maL_A (\maE, \maE')$ by $\maL_A(\maE)$ and this is then a unital $C^*$-algebra.

An example of adjointable operator is given by finite rank operators and $A$-compact operators that we recall now for convenience. A rank one operator $\theta_{v,v'}$ between $\maE$ and $\maE'$ for nonzero elements $v\in \maE$ and $v'\in \maE'$, is defined by the formula
$$
\theta_{v,v'}(z):=v <v' ,z>, \quad \text{ for } z\in \maE.
$$
Then $\theta_{v,v'}\in \maL_A(\maE, \maE')$, and finite rank operators will be finite sums of such rank one operators. The closure in $\maL_A (\maE, \maE')$ of the subspace of finite rank operators is the space of $A$-compact operators.

Finally, recall that an operator $t$ from a Hilbert $A$-module $E$  to a Hilbert $A$-module $F$ is called regular if $t$ is closed and densely defined with densely defined adjoint $t^*$ such that  the operator $1+t^*t$ has dense image in $F$. See \cite{Pal} for more details.

\subsection{Regular HP complexes}

Unless otherwise specified, our $C^*$-algebras will always be unital.

\begin{definition}\label{HPcomplex}
\label{hpcomplexes}
An $n$-dimensional Hilbert-Poincar\'{e} complex (abbreviated HP-complex) over a $C^*$-
algebra $A$ is a complex $(\maE,b,T)$ of countably generated Hilbert right $A$-modules
$$
\maE_0 \xrightarrow{b_0} \maE_1 \xrightarrow{b_1} ...\xrightarrow{b_n} \maE_n
$$
where each $b_i$ is a densely defined closed regular operator with a densely defined regular adjoint $b^*_{\bullet}: \maE_{\bullet+1}\rightarrow \maE_{\bullet}$ such that successive operators in the complex are composable (i.e. the image of one is contained in the domain of the other) and $b_{i+1}\circ b_i=0$, together with adjointable operators $T:\maE_\bullet\rightarrow \maE_{n-\bullet}$
satisfying the following properties:
\begin{enumerate}
\item For $v\in \maE_p$, $T^*v=(-1)^{(n-p)p}Tv$.
\item $T$ maps $Dom(b^*)$ to $Dom(b)$, and we have for $v \in Dom(b^*)\subset \maE_p$, $$Tb_{p-1}^*v+(-1)^{p+1}b_{n-p}Tv=0.$$
\item $T$ induces an isomorphism between the cohomology of the complex
$(\maE,b)$ and that of the dual complex $(\maE,b^*)$:
$$\maE_n \xrightarrow{b^*_{n-1}} \maE_{n-1} \xrightarrow{b^*_{n-2}} \cdots \xrightarrow{b^*_0} \maE_0 $$
i.e. the induced map $T_*: H^k(\maE,b)\rightarrow H^k(\maE,b^*)$ is an isomorphism.
\item The  operator $B:= b + b^*: \maE\rightarrow \maE$ is densely defined and  satisfies that $B\pm i$ are injective with dense images and that 
$(B\pm i)^{-1}\in \mathcal{K}_A(\mathcal{E})$.
\end{enumerate}

\end{definition}\label{S2=1}
We have denoted by $\maE$ the direct sum $\oplus_{0\leq i \leq n} \maE_i$, by $b= \oplus_{0\leq i \leq n} b_i$ and similarly for $b^*$.  
The fourth item implies that $B$ is closed and is moreover a regular Fredholm operator, i.e. it has an inverse  modulo compact operators in the direct sum Hilbert module. 
\begin{remark}
Notice that if the {\em {duality operator}} $T$ satisfies the extra-condition $T^2=\pm 1$ then condition (2) implies condition (3) of Definition \ref{HPcomplex}.
Notice also that in this case, $T$ also induces $T_*:H^k(\maE,b^*)\rightarrow H^k(\maE,b)$.

\end{remark}

 It is worth pointing out for further use that a regular operator $t$ is Fredholm if and only if it has a pseudo-left inverse and a pseudo-right inverse. A pseudo-left inverse for $t$ is an operator $G\in \maL_A(\maE)$ such that $Gt$ is closable, $\overline{Gt}\in \maL_A(\maE)$, and $\overline{Gt}= 1 \text{ mod } \maK_A(\maE)$. Similarly a pseudo-right inverse for $t$ is an operator $G'\in \maL_A(\maE)$ such that $tG'$ is closable, $\overline{tG'}\in \maL_A(\maE)$, and $\overline{tG'}= 1 \text{ mod } \maK_A(\maE)$. The cohomology of the complex $(\maE,b)$ is defined here to be the unreduced one given by
$$
H^k(\maE,b):= \frac{\Ker (b_k)}{\Im (b_{k-1})}.
$$

\begin{remark}
The complex $(\maE,b,T)$ given in the definition is understood as a two-sided infinite complex with finitely many non-zero entries.
\end{remark}

\begin{definition} [\cite{HigsonRoeII}, Definition 5.4]\label{symm}
Let $dim \maE=n=2l+1$ be odd. Define on $\maE_p$,
$$
S= i^{p(p-1)+l}T \text{ and } D=iBS.
$$
Then we call $D$ the signature operator of the HP-complex $(\maE,b,T)$.
\end{definition}

We have the following proposition from \cite{HigsonRoe}, which is valid in our setting.

\begin{proposition}[\cite{HigsonRoe}, Lemma 3.4]
With the above notations we have $S^*=S$ and $bS+Sb^*=0$.
\end{proposition}

Recall that a regular operator $t$ is adjointably invertible if there exists an adjointable operator $s$ such that $st\subseteq ts=1$. Notice that when $t$ is self-adjoint, this is equivalent to the surjectivity of $t$, see \cite{Kustermans}. We now extend Proposition 2.1 in the first paper of \cite{HigsonRoe} to our setting. By definition, an HP-complex is acyclic if its cohomology groups are all zero.

\begin{proposition}
An HP complex is acyclic if and only if the operator $B$ is adjointably invertible.
Moreover, in this case $B^{-1}\in \mathcal{K}_{\mathcal{A}}(\mathcal{E})$.
\end{proposition}

\begin{proof}\
We adapt the proof given in \cite{HigsonRoe}. Assume that $B$ is invertible. Then for $v\in \Ker(b)$, there exists $w\in Dom(B)$ such that $v=Bw$, and
$$
\| b^*w\|^2=\|<b^*w,b^*w>\| = \|<w,bb^*w>\| = \|<w,bBw>\| =0,
$$
and hence $v= Bw=bw \in \Im(b)$. Therefore $\Ker(b)=\Im(b)$ and thus the complex is acyclic.

Conversely, let the HP-complex be acyclic. To prove that $B$ is adjointably invertible it suffices to prove that $B$ is surjective. Since all the cohomologies are trivial, $\Im(b)=\Ker(b)$, so the range of $b$ is closed.
Since the differentials $b_k, k= 0,1...,n$ are regular operators,  $Q(b)= b(1+b^*b)^{-1/2}$ is a bounded adjointable operator and we have $\Im(b)= \Im(Q(b))$, $\Ker(b)= \Ker(Q(b))$. Then by the Open Mapping Theorem, $Q(b)Q(b^*)$ is bounded below on $\Im(Q(b))$ and therefore $\Im(Q(b))\subseteq \Im(Q(b)Q(b^*))$, see Theorem 3.2 in \cite{Lance} . Similarly, $\Im(Q(b^*))\subseteq \Im(Q(b^*)Q(b))$.

Now, as $\Im(Q(b))$ is closed, $\Ker(Q(b))$ is an orthocomplemented submodule with $\Ker(Q(b))^\perp = \Im(Q(b)^*)= \Im(Q(b^*))$. Hence we have $\maE = \Im(Q(b))\oplus \Im(Q(b^*))$. So for any $v\in \maE$, we have $v= Q(b)v_1+Q(b^*)v_2$ for some $v_1,v_2\in \maE$. However $Q(b)v_1= Q(b)Q(b^*)w_1$, and $Q(b^*)v_2= Q(b^*)Q(b)w_2$  for some $w_1,w_2 \in \maE$. Hence we have for any $v \in \maE$,
$$
v = Q(b)Q(b^*)w_1+ Q(b^*)Q(b)w_2.
$$
But we have by Lemma \ref{Q(b)} below
$$
Q(b)^2 =Q(b^*)^2=0 \text{ and } Q(b+b^*)=Q(b)+Q(b^*).
$$
So  we have
$$
v= (Q(b)+Q(b^*))(Q(b^*)w_1+Q(b)w_2),
$$
which shows that $Q(b)+Q(b^*)$ is surjective and hence so is $Q(b+b^*)$. However, $\Im(Q(b+b^*))=\Im(B)$ and hence $B$ is surjective and thus invertible.
\end{proof}

\begin{lemma}\label{Q(b)}\
 We have
  $$
  Q(b)^2 =Q(b^*)^2=0 \text{ and } Q(b+b^*)=Q(b)+Q(b^*).
  $$
\end{lemma}

\begin{proof}
A classical argument gives (see also \cite{SkandalisCourse}, \emph{Remarques (a)}, page 74)
$$
\Ker (Q(b)) = \Ker (b) \text{ and } \Im (Q(b))=\Im (b).
$$
This shows that $Q(b)^2 =0$. The same argument applied to $b^*$ gives $Q(b^*)^2=0$.

Let now $f= Q(b)$. Then we have $b= f(1-f^*f)^{-1/2}$, $(1+b^*b)^{-1/2}= (1-f^*f)^{1/2}$, and since $fp(f^*f)= p(ff^*)f$ for any polynomial $p$, by continuity it also holds for any $p\in C([0,1])$. So in particular we have
\begin{equation}\label{propf1}
f(1-f^*f)^{1/2} = (1-ff^*)^{1/2} f.
\end{equation}
Notice now that the operators $(1+b^*b)(1+bb^*)$, $(1+bb^*)(1+b^*b)$ and $1+(b+b^*)^2$, defined on
$\Dom (b^*b)\cap\Dom (bb^*)$, coincide. Therefore, setting $A=(1+b^*b)^{-1/2}$ and $B=(1+bb^*)^{-1/2}$ and using the spectral theorem, it is easy to see that we have the
following equality
of unbounded operators
$$
AB = BA = [(1+b^*b)(1+bb^*)]^{-1/2}.
$$
Hence
$$
Q(b+b^*) = (b+b^*)[1+(b+b^*)^2]^{-1/2}=(b+b^*) [(1+b^*b)(1+bb^*)]^{-1/2} = (b+b^*) AB = bAB +b^*BA=fB+f^*A.
$$
We note that for any polynomial $p$ with $p(0)=1$, we have $fp(ff^*)=f$, since $f^2=0$. So the equality also holds by continuity for any $p\in C([0,1])$ for which $p(0)=1$. In particular we note  that
\begin{equation}\label{propf2}
   fB= f(1-ff^*)^{1/2}=f.
\end{equation}
Similarly $f^*A=f^*$ and we finally get
$$
Q(b+b^*) = fB+f^*A=f+f^*= Q(b)+Q(b)^*=Q(b)+Q(b^*).
$$
The above computation is justified by the facts that $\Im(b) \subseteq Dom(b)$ and $\Im(1+b^*b)^{-1/2} \subseteq Dom(b)$.

\end{proof}

\begin{proposition}[\cite{HigsonRoe}, Lemma 3.5]
The self-adjoint operators $B\pm S: \maE\rightarrow \maE$ are invertible.
\end{proposition}

\begin{proof}\ Consider the mapping cone complex of the chain map $S: (\maE,b)\rightarrow (\maE, -b^*)$ where we have changed $b^*$ to $-b^*$. Its differential is
$$
d_S=\left(
\begin{array}{ccc}
 b & 0\\
 S & b^*
 \end{array}\right)
$$
Since $S$ is an isomorphism on cohomology, its mapping cone complex is acyclic, i.e. all the cohomology groups are zero. Therefore the operator $B_S=d_S + d_S^*$ is invertible on $\maE\oplus \maE$. Now, $B_S= \left(
\begin{array}{ccc}
B& S\\
 S & B
 \end{array}\right)$
preserves the $+1$ eigenspace of the involution which interchanges the copies of $\maE$  and    identifies on this subspace with $B+S$ on $\maE$. hence $B+S$ is invertible.
Considering the $-1$ eigenspace similarly, we deduce that and $B-S$ is also invertible.
\end{proof}

\begin{definition}
Let $(\maE,b,T)$ be an odd-dimensional Hilbert-Poincar\'{e} complex and $S$ be the operator associated to
$T$ as in Definition \ref{symm}. Then the signature of $(\maE,b,T)$ is defined as the class of the self-adjoint
invertible operator $(B+S)(B-S)^{-1}\in K_1(\mathcal{K}_A(\maE))$. We denote this class by $\sigma(\maE,b,T)$.
\end{definition}

Now, let $A,B$ be $\sigma$-unital $C^\ast$-algebras which are Morita-equivalent, with Morita bimodule $\leftsub{A}{\maE}_B$. So we have $A\cong \mathcal{K}_B(\leftsub{A}{\maE}_B)$, and so there is a $\ast$-homomorphism $\phi: A\rightarrow \mathcal{L}(\leftsub{A}{\maE}_B)$.
Let now $(\maE,b,T)$ be a Hilbert-Poincar\'{e} complex of countably generated $A$-modules.
For the de Rham complexes which admit the Hodge $*$-operator, the operator $S$ is an involution.
In this case, it is easy to relate $\sigma(\maE, b, T)$ with the usual definition of the signature for
the operator $D$, i.e.
$$
(D+iI)(D-iI)^{-1}=(iBS+iI)(iBS-iI)^{-1}=(B+S)SS^{-1}(B-S)^{-1}=(B+S)(B-S)^{-1}
$$
We assume for the rest of this subsection that  $S^2=1$. We then form a Hilbert-Poincar\'{e} complex $(\maE\otimes \leftsub{A}{\maE}_B,b\otimes I)$:

$$\maE_0\otimes\leftsub{A}{\maE}_B \xrightarrow{b\otimes I} \maE_1\otimes \leftsub{A}{\maE}_B \xrightarrow{b\otimes I} \maE_2\otimes\leftsub{A}{\maE}_B...\xrightarrow{b\otimes I} \maE_n\otimes\leftsub{A}{\maE}_B$$

Let $\mathcal{M}:K_1(A)\rightarrow K_1(B)$ be the isomorphism induced by the Morita equivalence between $A$ and $B$.

\begin{proposition} With the above assumptions we have,
$\mathcal{M}[\sigma(\maE,b,T)]= \sigma(\maE\otimes \leftsub{A}{\maE}_B,b\otimes I,T\otimes I)$.
\end{proposition}

\begin{proof}
Let $U(D)=(D+iI)(D-iI)^{-1}$ and $\maE_A:= \oplus_p \maE_p$ then $U(D)=(B+S)(B-S)^{-1}$. The class of $U(D)$ in $K_1(A)$ can be identified with the class of the $KK$-cycle in $KK(\mathbb{C},A)$ given by $(\maE_A, \lambda,U(D))$, where $\lambda$ is the scalar multiplication by complex numbers on the left. Then, $\mathcal{M}[\sigma(\maE,b,T)]$ can be identified with an element of $KK(\mathbb{C},B)$ which will be given by the Kasparov product of $(\maE_A, \lambda,U(D))$ with the Morita $KK$-cycle $(\leftsub{A}{\maE}_B,\phi,0)$.

But this Kasparov product is given by the $KK$-cycle $(\maE_A\otimes\leftsub{A}{\maE}_B,\lambda\otimes I,U(D)\otimes I)$ \cite{ConnesSkandalis, Kasparov}. Since $D$ is  self-adjoint regular operator, by the uniqueness of the functional calculus we have $U(D)\otimes I= U(D\otimes I)$. But then we can identify the class of $U(D\otimes I)$ in $K_1(B)$ with the cycle $(\maE_A\otimes\leftsub{A}{\maE}_B,\lambda\otimes I,U(D)\otimes I)$ in $KK(\mathbb{C},B)$. This finishes the proof.

\end{proof}

\subsection{Homotopy of HP-complexes}

We can now define the notion of homotopy equivalence of HP-complexes

\begin{definition}\label{homeqHPC}
A homotopy equivalence between two HP-complexes $(\maE,b,T)$ and $(\maE',b',T')$ is a chain map $A: (\maE,b)\rightarrow (\maE',b')$ which induces an isomorphism on cohomology and for which the maps $ATA^*$ and $T'$ between the complex $(\maE',b')$ and its dual $(\maE',b'^*)$ induce the same map on cohomology.
\end{definition}

\begin{definition}\
Let $(\maE,b)$ be a complex of Hilbert-modules. An \emph{operator homotopy} between Hilbert-Poincar\'{e} complexes $(\maE,b,T_1)$ and $(\maE,b,T_2)$ is a norm-continuous family of adjointable operators $T_s, s\in [0,1]$ such that each $(\maE,b,T_s)$ is a Hilbert-Poincar\'{e} complex.
\end{definition}

\begin{theorem}[\cite{HigsonRoe}, Theorem 4.3]\label{homsign}
If two odd-dimensional HP-complexes $(\maE,b,T)$ and $(\maE',b',T')$ are homotopy equivalent then their signatures are equal in $K_1(\mathcal{K}_A(\maE))$.
\end{theorem}

\begin{proof}\
The proof  given in \cite{HigsonRoe} works word by word in this case. Namely, it is shown that the signature of the complex $(\maE\oplus \maE', b\oplus b', T\oplus -T')$ is zero. This is achieved by using the chain map $A$ in the definition of homotopy equivalence to construct an explicit path that connects the operator $T\oplus -T'$ to an operator which is in turn operator homotopic to its additive inverse. More precisely,  the operator path is given very briefly as follows:
\begin{itemize}
\item First, the operator path
$$
\left(\begin{array}{ccc} T & 0 \\
0 & (s-1)T'-sATA^* \end{array}\right), 0\leq s\leq 1
$$
connects the duality operators for the direct sum HP-complex $T\oplus -T'$ to $T\oplus -ATA^*$.
\item Next, the operator $T\oplus -ATA^*$ is connected to $\left(\begin{array}{ccc} 0 & TA^* \\
AT & 0 \end{array}\right)$ via the path
$$\left(\begin{array}{ccc} \cos(s) T & \sin(s)TA^* \\
\sin(s)AT & -\cos(s)ATA^* \end{array}\right), 0\leq s \leq \frac{\pi}{2}.
$$
\item  Finally, the operator $\left(\begin{array}{ccc} 0 & TA^* \\
AT & 0 \end{array}\right)$ is connected to its additive inverse using the path
$$
\left(\begin{array}{ccc} 0 & \exp{is}TA^* \\
\exp{-is}AT & 0 \end{array}\right), 0\leq s \leq \pi.
$$
\end{itemize}
Thus using Lemma \ref{lem2} proved below and the fact that $\sigma(\maE\oplus \maE', b\oplus b', T\oplus -T')= \sigma(\maE, b, T)-\sigma(\maE',b', T')$, the proof is complete.
\end{proof}

So, we need to prove the following

\begin{lemma}[\cite{HigsonRoe}, Lemma 4.5]\label{lem1}
Operator homotopic HP-complexes have the same signature.
\end{lemma}

\begin{proof}
We adapt the proof of \cite{HigsonRoe} to our setting.

Let $(\maE,b)$ be a complex of Hilbert-modules and $T_s, s\in [0,1]$ be a norm-continuous family of duality operators acting on $(\maE,b)$ and $S_s$ be the self-adjoint operators defined from $T_s$ as in definition of the operator $S$. First we note from Result 5.22 in \cite{Kustermans} that for a regular operator $ t$ the map $ \C \supseteq\rho(t)\ni \lambda\mapsto (t-\lambda )^{-1}$ is continuous.
Since $(B+ S)$ is an invertible self-adjoint regular operator, the path of operators $(B+S+i\mu)^{-1}$, is norm continuous for $\mu\in [0,1]$.
Now for a fixed $\mu \in \R$ and any $s_1,s_2\in \R$, the resolvent identity holds:
$$
(B+S_{s_1}+i\mu)^{-1}-(B+S_{s_2}+i\mu)^{-1} =  (B+S_{s_1}+i\mu)^{-1}(S_{s_2}-S_{s_1})(B+S_{s_2}+i\mu)^{-1}.
$$
One can use techniques in Theorem VI.5 of \cite{ReedSimon} to show that the above identity implies that $(B+S_s+i\mu)^{-1}$ is norm-continuous in $s\in [0,1]$.

Now, consider the norm continuous path $(B+S_0+is)(B+S_0+is)^{-1}$ that can be concatenated with the norm continuous path $(B+S_s+i)(B+S_s+i)^{-1}$
and then again concatenated with the norm continuous path $(B+S_1+(1-s)i)(B+S_1+(1-s)i)^{-1}$. This yields a a well defined norm continuous path joining
$(B+S_0)(B-S_0)^{-1}$ and $(B+S_1)(B-S_1)^{-1}$. Therefore, these operators lie in the same $K_1$-class.  We thank the referee for suggesting this concatenation of
paths in place of our first more complicated proof.
\end{proof}

\begin{remark}\label{lem2}
From the proof of the previous lemma, it is easy to check that if the duality operator $T$ is operator homotopic to $-T$ then the signature of the HP-complex is zero. See [\cite{HigsonRoe}, Lemma 4.6] for more details.
\end{remark}

\subsection{The leafwise de Rham HP-complex}

We are mainly interested in HP-complexes arising from the study of homotopy invariants constructed out of the signature operator on smooth foliations \cite{Chang, HilsumSkandalis, KaminkerMiller, CW, Weinberger, Mathai, PiazzaSchick, HeitschLazarov}, and we proceed now to explain this ''paradigm example''. Let then $(V,\mathcal{F})$ be an oriented smooth foliation on a closed Riemannian manifold $(V, g)$. The leafwise tangent space $T\maF$ is then endowed with a euclidean structure which allows to induce the complex Grassmann bundles $\Lambda^iT^*\maF$ with hermitian structures. Assume that the dimension of $V$ is $n$ and that the dimension of the leaves is $p$ and set $q=n-p$ for the codimension of the foliation. We restrict to odd dimensional foliations  as this is not as well understood as is the even dimensional situation, see for instance \cite{BenameurHeitschI, BenameurHeitschJDG, LawsonMichelson, MischenkoFomenko, Neumann} where higher signatures play a fundamental part in the even case. Denote by $\mathcal{G}$ the monodromy (we could as well use holonomy) groupoid of the foliation and let $\lambda = (\lambda_x)_{x\in V}$ be a right-invariant smooth Haar system on $\mathcal{G}$.  The space  $\cG^{(1)}$ of arrows is the space of homotopy classes of paths drawn in the leaves  of $(V,\maF)$ and we make as usual the convenient confusion between $\maG^{(1)}$ and $\maG$. So, two paths whose ranges are contained in a given leaf $L$ define the same class in $\cG$ if they start and end at the same points and if they are homotopic through paths drawn in the same leaf $L$ and with fixed end points. Notice that  concatenation of paths endows $\cG$ with the structure of a smooth groupoid, which is also  a foliated manifold. We denote as usual by $s:\cG\to V$ and $r:\cG\to V$ the source and range maps and we use the following standard notation. For subsets $X, Y$ of the manifold $V$, we set
 $$
\mathcal{G}_X:= s^{-1}(X), \mathcal{G}^Y:= r^{-1}(Y), \mathcal{G}^Y_X:= r^{-1}(Y)\cap s^{-1}(X).
$$
Notice that when $Y=X$, the subspace $\mathcal{G}^X_X$ is a subgroupoid of $\cG$. Let $X$ be a complete smooth transversal of the foliation. The subspace $\maG_X$  is  a smooth submanifold of $\maG$ which is foliated by the pull-back foliation $\maF_X$ under the range map $r:\maG_X\to V$. We set $\mathcal{E}^i_c:=C_c^\infty(\mathcal{G}_X,r^* \Lambda^i T^*\mathcal{F}_X)$ with the $\mathcal{A}_c^X:=C_c^\infty(\mathcal{G}_X^X)$ valued inner product given by the following formula:
\begin{equation}
\label{innprod} <\xi_1,\xi_2>(u)= \int_{v \in \mathcal{G}_{r(u)}}
<\xi_1(v),\xi_2(vu)> \; d\lambda_{r(u)}(v) \text{
for $\xi_1,\xi_2 \in \mathcal{E}^i_c, u \in \mathcal{G}_X^X$}
\end{equation}
The space $\mathcal{E}^i_c$ is a right $\mathcal{A}_c^X$-module and the  formula for this action is given by
\begin{equation}
\label{rtact} (\xi f)(\gamma)= \sum_{\gamma' \in
\mathcal{G}_{r(\gamma)}^X} f(\gamma' \gamma) \xi(\gamma'^{-1}) \text{ for $f \in \mathcal{A}_c^X , \xi \in \mathcal{E}^i_c, \gamma \in
\mathcal{G}_X$}
\end{equation}
A classical computation shows that $\maE^i_c$ is then a preHilbert module over the  pre-$C^*$-algebra $\mathcal{A}_c^X$. By taking the completion of $\mathcal{A}_c^X$ with respect to the maximal $C^*$-norm and then completing the above pre-Hilbert module we obtain a Hilbert $C^*(\mathcal{G}^X_X)$-module $\mathcal{E}_i$ for $0\leq i \leq p=\dim \mathcal{F}$.

Consider now the leafwise de Rham differential $d=(d_x)_{x\in V}$ on $(V,\mathcal{F})$ and for each $x\in V$ denote the $\mathcal{G}_x^x$- equivariant lift of $d_x$ to $\mathcal{G}_x$ by $\tilde{d}_x$. Let $\tilde{d}$ denote the family of operators $(\tilde{d}_x)_{x\in V}$ acting on $\mathcal{E}^i_c$. Then  $\tilde{d}^2=0$ and  we get the de Rham complex on $\mathcal{G}_X$:
 $$
  \mathcal{E}^0_c\xrightarrow{\tilde{d}} \mathcal{E}^1_c\xrightarrow{\tilde{d}} \cdots \xrightarrow{\tilde{d}} \mathcal{E}^p_c
  $$
  The operator ${\tilde d}$ is thus a densely defined (unbounded) operator from  $\maE^i$ to $\maE^{i+1}$ which obviously extends to a {\em {closed}}  operator that we denote by $d_X$. 
%

We also consider the leafwise Hodge $\star$ operator along the leaves of $(V,\mathcal{F})$ associated with the fixed orientation of $T\maF$, and denote its lift to $\mathcal{G}_X$ by
$$
T_X: C_c^\infty(\mathcal{G}_X,r^*(\Lambda^i T^*\mathcal{F}) \longrightarrow  C_c^\infty(\mathcal{G}_X,r^*(\Lambda^{p-i} T^*\mathcal{F})).
$$
\begin{proposition}
The complex $(\maE, d_X)$
$$
\mathcal{E}_0\xrightarrow{d_X}\mathcal{E}_1\xrightarrow{d_X} \mathcal{E}_2\cdots \xrightarrow{d_X} \mathcal{E}_p,
$$
together with the operator $T_X$ is an HP-complex over the $C^*$-algebra $C^*(\maG_X^X)$.
\end{proposition}

\begin{proof}
Since $T_X$ is the lift of the leafwise Hodge $\star$ operator, we have on smooth compactly supported forms:
$$
\left<T_X \omega_1, T_X\omega_2\right> = \left<\omega_1,\omega_2\right>\text{   and } T_X(T_X \omega)= (-1)^{k(n-k)} \omega\quad \text{ for }\omega \in \mathcal{E}_c^k.
$$
Hence we get $T_X^*T_X=1$ and $T^*_X = (-1)^{k(n-k)} T_X$ on $\mathcal{E}^k_c$. Therefore $T_X$ extends to an adjointable operator on $\mathcal{E}_k$ which satisfies (1) of Definition $\ref{hpcomplexes}$. 
The adjoint of $\tilde{d}$ is easily seen to be the operator $\tilde{\delta}: \mathcal{E}^{i}_c\rightarrow \mathcal{E}_c^{i-1}$ given by the formula
$$
\tilde{\delta} := (-1)^{p(i+1)+1} T_X \tilde{d} T_X.
$$
Then $\tilde{\delta}$ extends to a closed densely defined $C^*(\maG_X^X)$-linear operator $\delta_X: \mathcal{E}_{\bullet} \rightarrow \mathcal{E}_{\bullet-1}$. We then have  $T_X\circ {\delta}_X=(-1)^k d_X\circ T_X$ on smooth $k$-forms, and hence as closed operators on the maximal completions. This shows condition (2) of Definition $\ref{hpcomplexes}$.

To see that the third condition is verified, we first note that due to condition (2) and Remark \ref{S2=1}, the map $T_X$ takes $\Im(\delta_X)$ to $\Im(d_X)$ and therefore the induced map
$(T_X)_*: H^*(\maE,b)\rightarrow H^*(\maE,b^*)$ is well-defined.

Let $z\in \mathcal{E}_k$ be $d_X$-closed and such that $[T_X z] =0 \in H^{n-k}(\maE,\delta_X)$. Then there exists a $y \in Dom (\delta_X)\subset \mathcal{E}_{n-k+1}$ such that $T_X z=  \delta_X y$ and we have
$$
z= \pm T_X(  \delta_X y) = d_X (\pm T_X y).
$$
Thus $z\in \Im(b)$. Therefore the induced map $(T_X)_*$ is injective. Surjectivity of $(T_X)_*$ follows from the relations  $T_X\circ {\delta}_X=(-1)^k d_X\circ T_X$ and
$T_X^2=\pm 1$. Hence $(T_X)_*$ is an isomorphism.

Finally, to check condition (4) in Definition $\ref{hpcomplexes}$ we remark that $\tilde{d}+ \tilde{\delta}$ is an elliptic $\maG$-operator and therefore extends to a regular Fredholm operator on the Hilbert module, and the extension of $\tilde{d}+\tilde{\delta}$ coincides with $d_X+\delta_X$ (cf. \cite{Vassout, VassoutThesis}). Moreover,
since $(\tilde{d}+\tilde{\delta}\pm i)^{-1}$ is a pseudo-differential $\mathcal{G}$-operator of negative order, its extension to the Hilbert module
is a compact operator \cite{ConnesIntegration}. This extension coincides again with $(d_X+\delta_X\pm i)^{-1}$.
%
%
%
\end{proof}

\begin{proposition}
The closed unbounded operators $d_X$  and $\delta_X$ are regular operators.
\end{proposition}

\begin{proof}
This is a classical result that we sketch for completeness, see for instance \cite{HilsumSkandalisI}.
Let $\Delta= \tilde{d}\tilde{\delta}+\tilde{\delta}\tilde{d}$ on $\mathcal{E}^k_c$. Then we have on
$\mathcal{E}^k_c$:
$$
(1+\tilde{d}\tilde{\delta})(1+\tilde{\delta}\tilde{d})=(1+\tilde{d}\tilde{\delta}+\tilde{\delta}\tilde{d})=(1+\Delta).
$$
Moreover, the leafwise elliptic operator $\Delta$ extends to a regular operator $\Delta_X$ on the corresponding $2$-Sobolev space
$Dom(\Delta_X)$ on which $(1+\Delta_X)$ is surjective, see for instance \cite{Vassout}. Since the Sobolev space $Dom(\Delta_X)$ is contained in
$Dom(d_X\delta_X)\cap Dom(\delta_Xd_X)$, we deduce that for any element $z\in Dom(\Delta_X)$, the element $z+\delta_Xd_Xz$ makes sense and
 belongs to $Dom(d_X\delta_X)$ since $\Im (\delta_X)\subset \Ker(\delta_X)$. So, we deduce that for any $t\in \maE$, there exists $z\in Dom(\Delta_X)$ such that
$z+\Delta_X z=t$ and considering $u=z+\delta_Xd_Xz$, we deduce that  $u$ belongs to $Dom(d_X\delta_X)$ and satisfies $u+d_X\delta_X u = t$. This shows that
$1+d_X\delta_X$ is surjective and hence that $d_X$ is regular. The similar proof works for $\delta_X$.

\end{proof}

\section{Hilbert modules and leafwise homotopy equivalence}\label{MoritaModules}

We review in this section some classical properties of Hilbert modules  associated with leafwise maps that will be used in the subsequent sections. We fix
 two smooth foliations $(V,\maF)$ and $(V', \maF')$ together with a leafwise map $f:(V,\mathcal{F})\rightarrow (V',\mathcal{F}')$. So $f$ is a smooth map
which sends leaves to leaves. Denote by $\mathcal{G}$ and $\mathcal{G}'$ the monodromy  groupoids of $(V,\mathcal{F})$ and  $(V',\mathcal{F}')$,
respectively.
The leafwise map $f$ naturally induces a well-defined map still denoted $f:\mathcal{G}\rightarrow \mathcal{G}'$ which is clearly a groupoid morphism.
In the sequel and for simplicity, we will use the same notation $r$ and $s$ for the range and the source maps on the groupoids $\mathcal{G}$ and
$\mathcal{G}'$.
We are only interested in leafwise homotopy equivalences, we shall  therefore make the following simplifying assumption\\

\begin{assumption}\label{Etale}
  For any leaf $L'$ of $(V', \maF')$ and any transverse submanifold $X$ to $(V, \maF)$, the intersection $f^{-1} (L')\cap X$ is (at most) a countable subset
of $X$.
\end{assumption}

Notice that Assumption \ref{Etale} is satisfied when $f$ satisfies that $f^{-1} (L')$ is a finite union of leaves of $(V, \maF)$, for any given leaf $L'$
 of $(V', \maF')$. For a leafwise homotopy equivalence, this inverse image is a single leaf. In the whole present section,
leafwise map means smooth leafwise map satisfying Assumption \ref{Etale}.

\subsection{The reduced Hilbert bimodule of a leafwise map}

 We now introduce the reduced graph $(\maG_{W'}^W(f), r_f, s_f)$ associated with the subspaces $W$ and $W'$ of $V$ and $V'$ respectively, by setting

\hspace{0.25in}
\begin{picture}(415,80)

\put(155,60){$ \mathcal{G}^W_{W'}(f):= \{(w,\gamma') \in W\times \mathcal{G}'_{W'}, f(w)= r(\gamma')\} $}
\put(210,50){ $\vector(-4,-3){30}$}
\put(180,40){$r_f$}

\put(170,15){$W$}

\put(285,15){$W'$}
\put(245,50){ $\vector(4,-3){30}$}
\put(280,40){$ s_f$}
\end{picture}

where $ r_f(w,\gamma')= w\text{ and } s_f(w,\gamma')=s(\gamma')$.
Let $X$ (resp. $X'$) be a complete smooth transversal in $(V,\mathcal{F})$ (resp. in $(V',\mathcal{F}')$). We shall be mainly interested in the case $W=X$ and $W'=X'$ and in the reduced graph $\maG_{X'}^X(f)$. The groupoid ${\mathcal{G}'}_{X'}^{X'}$ acts on the right on $\mathcal{G}^X_{X'}(f)$ as follows. If $\alpha'\in {\mathcal{G}'}^{X'}_{X'}$ is such that $r(\alpha')=s_f(x, \gamma')$, then $(x,\gamma')\alpha' := (x,\gamma'\alpha')$. Note that $r(\gamma'\alpha')=f(x)$, so this action is well-defined. It is easy to see that since ${\mathcal{G}'}_{X'}^{X'}$ acts properly and freely on $\mathcal{G}'_{X'}$, it also acts properly and freely  $\mathcal{G}^X_{X'}(f)$.
%

The space $C_c(\mathcal{G}_{X'}^X(f))$, of compactly supported continuous complex valued functions on $\mathcal{G}_{X'}^X(f)$, is thus endowed with the structure of a right $C_c(\mathcal{G}'^{X'}_{X'})$-module. For $\xi \in C_c(\mathcal{G}_{X'}^X(f))$ and $\phi'\in C_c(\mathcal{G}'^{X'}_{X'})$ the module structure is defined by the formula
$$
(\xi\phi')(x,\gamma')=  \sum_{\alpha' \in \mathcal{G}'^{X'}_{s(\gamma')}} \xi(x, \gamma'\alpha'^{-1}) \;\; \phi'(\alpha').
$$
On the other hand, the groupoid $\mathcal{G}_X^X$ acts on the left on $\mathcal{G}_{X'}^X(f)$ through the formula
$$
\alpha (x,\gamma')= (r(\alpha),{f}(\alpha)\gamma') \quad\text{ for } (x,\gamma')\in \mathcal{G}_{X'}^X(f)\text{ and } \alpha\in \mathcal{G}_X^X.
$$
The left action of $\mathcal{G}_X^X$ on $\mathcal{G}_{X'}^X(f)$ induces a representation $\pi_f$ of the algebra  $C_c(\mathcal{G}_X^X)$  on the $C_c(\mathcal{G}_{X'}^{X'})$-module $C_c(\mathcal{G}_{X'}^X(f))$ given for $\phi\in  C_c(\mathcal{G}_X^X)$ and $\xi \in C_c(\mathcal{G}_{X'}^X(f))$ by the formula
$$
\pi_f(\phi)\xi(x,\gamma') = \sum_{\alpha\in \mathcal{G}_X^{x}} \phi(\alpha)\; \; \xi(s(\alpha),{f}(\alpha^{-1}) \gamma')
$$
We define the $C_c(\mathcal{G}'^{X'}_{X'})$-valued inner product  by the formula:
$$
<\xi,\eta>(\gamma') = \sum_{\gamma'_1\in \mathcal{G}'^{f(X)}_{r(\gamma')}}\;\; \sum_{\{x\in X, f(x)=r(\gamma'_1)\}} \overline{\xi(x,\gamma'_1)} \eta(x,\gamma'_1 \gamma'), \quad \text{ for any }\xi, \eta \in C_c(\mathcal{G}^X_{X'}(f)), \gamma'\in {\maG'}_{X'}^{X'}.
$$
Since $X$ is a transversal and by \ref{Etale}, the space $\{x\in X, f(x)=r(\gamma'_1)\}$ is a countable subset of the leaf $f^{-1} (L')$, where $L'$ is the leaf which contains (the representatives of) $\gamma'$. It is then easy to check, with obvious notations, that
$$
<\xi, \eta\phi'>=<\xi,\eta> \phi', <\xi,\eta>^\ast=<\eta,\xi> \text{ and } <\xi,\xi>\; \geq 0 \text{ in } C^\ast(\mathcal{G}'^{X'}_{X'}).
$$
%
%
Now, completing $C_c(\mathcal{G}^X_{X'}(f))$ with respect to the maximal $C^*$-algebra norm on $C_c(\mathcal{G}'^{X'}_{X'})$, we end up with a Hilbert $C^*$-bimodule over the maximal $C^*$-algebras $C^*(\mathcal{G}^{X}_{X})$ and  $C^*({\mathcal{G}'}^{X'}_{X'})$, that we denote  by $\mathcal{E}^X_{X'}(f)$.

\begin{remark}
The choice of maximal completion is dictated to us by the construction of  measured determinants and  rho invariants in Part II of this series of papers \cite{BenameurRoyII}. Similar results hold with other completions.
\end{remark}

Now, let $(V,X,\mathcal{F})$, $(V',X',\mathcal{F}')$ and $(V'',X'',\mathcal{F}'')$ be foliated manifolds with complete trasversals $X$, $X'$ and $X''$, respectively. Let
$$
(V,\mathcal{F}) \stackrel{f}{\rightarrow} (V',\mathcal{F}')\stackrel{g}{\rightarrow} (V'',\mathcal{F}''),
$$
be leafwise maps.  We define the space $\mathcal{G}_{X'}^X(f) \times_{\mathcal{G}'^{X'}_{X'}} \mathcal{G}_{X''}^{X'}(g)$ as the fibered product defined as the quotient of
$$
 \{((x,\gamma');(x',\gamma'')) \in \mathcal{G}_{X'}^X(f)\times \mathcal{G}_{X''}^{X'}(g),  x'=s(\gamma')\},
$$
under the equivalence relation
$$
((x,\gamma');(x',\gamma''))\sim((x,\gamma')\alpha';\alpha'^{-1}(x',\gamma'')), \text{ for } \alpha' \in \mathcal{G}'^{X'}_{X'}), r(\alpha')=s(\gamma')=x'.
$$
The equivalence class of $((x,\gamma');(x',\gamma''))$ is denoted $[(x,\gamma');(x',\gamma'')]$.

\begin{proposition}\label{composition}\
\begin{enumerate}
 \item The space $\mathcal{G}_{X'}^X(f)\times_{\mathcal{G}'^{X'}_{X'}} \mathcal{G}_{X''}^{X'}(g)$ is a smooth manifold which is diffeomorphic to $\mathcal{G}_{X''}^X(g\circ f)$.
\item The map $C_c(\maG_{X'}^X(f))\otimes_{C_c(\mathcal{G}'^{X'}_{X'})} C_c(\maG^{X'}_{X''}(g)\rightarrow  C_c(\maG^{X}_{X''}(g\circ f))$ which assigns to $\xi_f\otimes \eta_g$  for $\xi_f  \in C_c(\mathcal{G}_{X'}^X(f))$ and $\eta_g\in C_c(\mathcal{G}^{X'}_{X''}(g))$  the function $\xi_f \ast \eta_g$ given by
$$
\xi_f\ast \eta_g (x, \alpha'') := \sum_{\alpha'_1\in \cG_{X'}^{f(x)}} \xi_f (x, \alpha'_1) \eta_g (s(\alpha'_1), g({\alpha'}_1^{-1})\alpha''), \text{ for }(x, \alpha'')\in \mathcal{G}^{X}_{X''}(g\circ f).
$$
is well defined.
\end{enumerate}
\end{proposition}

\begin{proof}
\begin{enumerate}
\item We define a map $\chi: \mathcal{G}_{X'}^X(f)\times_{\mathcal{G}'^{X'}_{X'}} \mathcal{G}_{X''}^{X'}(g)\rightarrow \mathcal{G}_{X''}^X(g\circ f)$ in the following way:
$$
\chi([(x,\gamma');(x',\gamma'')])=(x,{g}(\gamma')\gamma'')
$$
We note that $g\circ f(x)= g(r(\gamma'))=r({g}(\gamma'))$, so $(x,g(\gamma')\gamma'')\in \mathcal{G}_{X''}^X(g\circ f)$. It is easy to see that this map is well defined and smooth since the map $\chi_0$ given by $\chi_0((x,\gamma');(x',\gamma''))=(x,{g}(\gamma')\gamma'')$ is clearly smooth. The relation
$$
\chi([(x_1,\gamma'_1);(x'_1,\gamma''_1)])= \chi([(x_2,\gamma'_2);(x'_2,\gamma''_2)]),
$$
implies that
$$
x_1=x_2 \text{ and } {g}(\gamma'_1)\gamma''_1={g}(\gamma'_2)\gamma''_2
$$
Now setting $\alpha'= \gamma'^{-1}_2 \gamma'_1 \in \mathcal{G}'^{X'}_{X'}$, we  get
$$
((x_2,\gamma'_2)\alpha';\alpha'^{-1}(x'_2,\gamma''_2)) = ((x_1,\gamma'_1);(x'_1,\gamma''_1)).
$$
So $\chi$ is injective. Surjectivity is also clear and uses that $X'$ is a complete transversal. The rest of the proof of the first item is also clear.

\item If we use the identification $\chi$ defined in the first item, then the formula for $\xi_f\ast \eta_g$ becomes
$$
\xi_f\ast \eta_g [(x,\gamma');(x',\gamma'')]:= \sum_{\alpha'\in \mathcal{G}'^{s(\gamma')}_{X'}} \xi_f(x,\gamma'\alpha')\eta_g(s(\alpha'),{g}(\alpha'^{-1})\gamma'').
$$
A direct inspection shows that $\xi_f\ast \eta_g$ is well defined on $\mathcal{G}^{X}_{X''}(g\circ f)$, is compactly supported and is continuous. Moreover, for $\phi'\in C_c(\mathcal{G}'^{X'}_{X'})$, we compute:
\begin{eqnarray}\label{lhs5}
\xi_f\phi'\ast \eta_g[(x,\gamma');(x',\gamma'')]&=&\sum_{\alpha'\in \mathcal{G}'^{s(\gamma')}_{X'}} (\xi_f\phi')(x,\gamma'\alpha')\eta_g(\alpha'^{-1}x',{g}(\alpha'^{-1})\gamma'')\nonumber\\
&=& \sum_{\alpha'\in \mathcal{G}'^{s(\gamma')}_{X'}} \sum_{\alpha'_1\in \mathcal{G}'^{X'}_{s(\alpha')}}(\xi_f)(x,\gamma'\alpha'\alpha'^{-1}_1)\phi'(\alpha'_1) \eta_g(\alpha'^{-1}x',{g}(\alpha'^{-1})\gamma'')
\end{eqnarray}
On the other hand, we also have:
\begin{eqnarray}\label{rhs5}
\xi_f\ast \pi_g(\phi')\eta_g[(x,\gamma');(x',\gamma'')]&=&\sum_{\beta'\in \mathcal{G}'^{s(\gamma')}_{X'}} \xi_f(x,\gamma'\beta')[\pi_g(\phi')\eta_g](\beta'^{-1}x',{g}(\beta'^{-1})\gamma'')\nonumber\\
&=& \sum_{\beta'\in \mathcal{G}'^{s(\gamma')}_{X'}} \xi_f(x,\gamma'\beta')[\pi_g(\phi')\eta_g](s(\beta'),{g}(\beta'^{-1})\gamma'')\nonumber\\
&=&\sum_{\beta'\in \mathcal{G}'^{s(\gamma')}_{X'}} \xi_f(x,\gamma'\beta') \sum_{\beta'_1\in  \mathcal{G}'^{s(\beta')}_{X'}}\eta_g(s(\beta'_1),{g}(\beta_1'^{-1}){g}(\beta'^{-1})\gamma'')\phi'(\beta'_1)\nonumber\\
&=&\sum_{\beta'\in \mathcal{G}'^{s(\gamma')}_{X'}} \xi_f(x,\gamma'\beta') \sum_{\beta'_2\in  \mathcal{G}'^{r(\beta')}_{X'}}\eta_g(s(\beta'_2),{g}(\beta_2'^{-1})\gamma'')\phi'(\beta'^{-1}\beta'_2)\nonumber\\
&=&\sum_{\beta'_2\in  \mathcal{G}'^{s(\gamma')}_{X'}}\sum_{\beta'\in \mathcal{G}'^{s(\gamma')}_{X'}} \xi_f(x,\gamma'\beta') \eta_g(s(\beta'_2),{g}(\beta_2'^{-1})\gamma'')\phi'(\beta'^{-1}\beta'_2)\nonumber\\
&=&\sum_{\beta'_2\in  \mathcal{G}'^{s(\gamma')}_{X'}}\sum_{\beta'_3\in \mathcal{G}'^{X'}_{s(\beta'_2)}} \xi_f(x,\gamma'\beta'_2\beta'^{-1}_3) \eta_g(s(\beta'_2),{g}(\beta_2'^{-1})\gamma'')\phi'(\beta'_3)
\end{eqnarray}
Comparing $\eqref{lhs5}$ and $\eqref{rhs5}$ gives $\xi_f\phi'\ast \eta_g =\xi_f\ast \pi_g(\phi')\eta_g.$

\end{enumerate}
\end{proof}
\begin{remark}
 We shall show in Proposition \ref{Iso}, under the simplifying assumption that $f$ is a leafwise homotopy equivalence   that the map defined in the previous proposition extends to
an isometric isomorphism of Hilbert modules.
\end{remark}

\begin{proposition}\label{Compacts}\
The representation $\pi_f$ is valued in the $C^*$-algebra   $\mathcal{K}_{C^*({\maG'}_{X'}^{X'})}(\mathcal{E}_{X'}^{X}(f))$ of adjointable compact operators:
$$
\pi_{f}:C^\ast(\mathcal{G}^{X}_{X})\longrightarrow \mathcal{K}_{C^*({\maG'}_{X'}^{X'})}(\mathcal{E}_{X}^{X'}( f)).
$$
Moreover, if $f$ is an oriented  leafwise homotopy equivalence, then $\pi_f$ is a $C^*$-algebra isomorphism.
\end{proposition}

\begin{proof}
 The first statement is clear since smooth compactly supported functions on $\cG_X^X$ yield compact operators of $\maE_{X'}^X(f)$ by classical arguments. Assume now that $f$ is an oriented  leafwise homotopy equivalence.
Let $\eta_1,\eta_2\in \mathcal{E}_{X}^{X'}(f)$ and denote by $\theta_{\eta_1,\eta_2}$ the corresponding compact operator of $\maE_{X'}^X(f)$, given by
$$
\theta_{\eta_1,\eta_2} \zeta := \eta_1 <\eta_2,\zeta>.
$$
The isomorphism of $C^*$-algebras follows from the fact that since the Hilbert-module $\maE_{X'}^X(f)$ is compatible with the Connes-Skandalis bimodule
(see Proposition \ref{compatibility}), it is an imprimitivity bimodule, using the corresponding result
from \cite{HilsumSkandalis}. It can also be proved directly and we proceed now to do it for surjectivity. The direct proof of injectivity is similar.
 A straightforward computation gives for $(x,\gamma')\in \cG_{X'}^X (f)$:
\begin{eqnarray*}
\theta_{\eta_1,\eta_2} \zeta\underbrace{(x,\gamma')}_{f(x)= r(\gamma')}&=& (\eta_1.<\eta_2,\zeta>)(x,\gamma')\nonumber\\
&=& \sum_{\alpha'\in \mathcal{G}'^{X'}_{s(\gamma')}} \eta_1(x,\gamma'\alpha'^{-1}) <\eta_2,\zeta>(\alpha')\\
&=& \sum_{\alpha'\in \mathcal{G}'^{X'}_{s(\gamma')}} \eta_1(x,\gamma'\alpha'^{-1}) \sum_{x_1\in X\cap L_{s(\alpha')} } \sum_{\gamma'_1\in \mathcal{G}'^{f(x_1)}_{r(\alpha')}} \overline{\eta_2(x_1,\gamma'_1)} \zeta(x_1,\gamma'_1\alpha') \\
&=& \sum_{\alpha'\in \mathcal{G}'^{X'}_{s(\gamma')}} \eta_1(x,\gamma'\alpha'^{-1}) \sum_{x_1\in X\cap L_{s(\alpha')} } \sum_{\gamma'_2\in \mathcal{G}'^{f(x_1)}_{s(\alpha')}} \overline{\eta_2(x_1,\gamma'_2\alpha'^{-1})} \zeta(x_1,\gamma'_2) \hspace{0.7cm}\\
&=& \sum_{\alpha'\in \mathcal{G}'^{X'}_{s(\gamma')}} \eta_1(x,\gamma'\alpha'^{-1}) \sum_{x_1\in X\cap L_{s(\gamma')} } \sum_{\gamma'_3\in \mathcal{G}'^{f(x)}_{f(x_1)}} \overline{\eta_2(x_1,\gamma'^{-1}_3\gamma'\alpha'^{-1})} \zeta(x_1,\gamma'^{-1}_3\gamma') \hspace{0.3cm}\\
&=& \sum_{x_1\in X\cap L_{s(\gamma')} } \sum_{\gamma'_3\in \mathcal{G}'^{f(x)}_{f(x_1)}} \sum_{\alpha'\in \mathcal{G}'^{X'}_{s(\gamma')}} \eta_1(x,\gamma'\alpha'^{-1})   \overline{\eta_2(x_1,\gamma'^{-1}_3\gamma'\alpha'^{-1})} \zeta(x_1,\gamma'^{-1}_3\gamma')\\
&=& \sum_{x_1\in X\cap L_{s(\gamma')} } \sum_{\gamma_3\in \mathcal{G}^{x}_{x_1}} \sum_{\alpha'\in \mathcal{G}'^{X'}_{s(\gamma')}} \eta_1(x,\gamma'\alpha'^{-1})   \overline{\eta_2(x_1,{f}(\gamma^{-1}_3)\gamma'\alpha'^{-1})} \zeta(s(\gamma_3),{f}(\gamma^{-1}_3)\gamma') \hspace{0.5cm}\\ &&\text{ (since $\exists$ unique $\gamma_3 \in  \mathcal{G}^{x}_{x_1}$ such that  ${f}(\gamma_3)=\gamma'_3$) }\\
&=& \sum_{\gamma_3\in \mathcal{G}^{x}_{X}} \sum_{\alpha'\in \mathcal{G}'^{X'}_{s(\gamma')}} \eta_1(x,\gamma'\alpha'^{-1})   \overline{\eta_2(s(\gamma_3),{f}(\gamma^{-1}_3)\gamma'\alpha'^{-1})} \zeta(s(\gamma_3),{f}(\gamma^{-1}_3)\gamma')\\
&=& \sum_{\gamma_3\in \mathcal{G}^{x}_{X}} \sum_{\alpha'_1\in \mathcal{G}'^{f(x)}_{X'}} \eta_1(x,\alpha'_1)   \overline{\eta_2(s(\gamma_3),{f}(\gamma^{-1}_3)\alpha'^{-1}_1)} \zeta(s(\gamma_3),{f}(\gamma^{-1}_3)\gamma')\\
&=& \sum_{\gamma_3\in \mathcal{G}^{x}_{X}}\left( \sum_{\alpha'_1\in \mathcal{G}'^{f(x)}_{X'}} \eta_1(r(\gamma_3),\alpha'_1)   \overline{\eta_2(s(\gamma_3),{f}(\gamma^{-1}_3)\alpha'^{-1}_1)}\right) \zeta(s(\gamma_3),{f}(\gamma^{-1}_3)\gamma')\\
&=& \sum_{\gamma_3\in \mathcal{G}^{x}_{X}} (\eta_1\star\eta_2)(\gamma_3) \zeta(s(\gamma_3),{f}(\gamma^{-1}_3)\gamma')\\
&=& \pi_f(\eta_1\star\eta_2)\zeta(x,\gamma')
\end{eqnarray*}
where we have denoted by $\eta_1\star\eta_2$ the function
$$
(\eta_1\star\eta_2) (\alpha) := \sum_{\alpha'_1\in {\cG'}_{X'}^{f(r(\alpha))}} \eta_1 (r(\alpha), \alpha'_1) {\overline{\eta_2(s(\alpha), f(\alpha^{-1})\alpha'_1)}}.
$$
Thus we get
$$
\theta_{\eta_1,\eta_2} = \pi_f (\eta_1\star \eta_2).
$$
This finishes the proof of surjectivity by classical arguments.
\end{proof}

\subsection{Pull-back maps on Hilbert modules}

Let as before $(V,X,\mathcal{F})$ and $(V',X',\mathcal{F}')$ be closed foliated (oriented) manifolds with complete transversals $X$ and $X'$, respectively.
Let again $f:(V,\mathcal{F})\rightarrow (V',\mathcal{F}')$ be a  leafwise oriented smooth leafwise map. The goal of the present section is to prove that
$f$ induces a well defined
 operator, the pull-back  $f_\phi^*$, which is   functorial and is moreover a chain map between the corresponding de Rham HP-complexes which, see Theorem  \ref{pullback} and Theorem \ref{isomorphism}. Let $E\to V$ and $E'\to V'$ be given hermitian vector bundles. Our  main  interest concerns  leafwise  Grassman bundles, over  $(V,\mathcal{F})$ and $(V',\mathcal{F}')$ respectively.

Since the $C^*$-algebras $C^*(\cG_X^X)$ and $C^*({\cG'}_{X'}^{X'})$ are only Morita equivalent, our goal is to define an adjointable homomorphism
$$
f_\phi^*: \maE_{X', E'} \longrightarrow \maE_{X, E} \otimes _{C^*(\cG_X^X)} \maE_{X'}^X (f),
$$
which will be associated with some nice cutoff function $\phi$ and which will be a chain map between the corresponding HP complexes \cite{HigsonRoe}. We later on prove a Poincar\'e Lemma when $f$ is an oriented  leafwise homotopy equivalence.

Notice that  the manifold $\mathcal{G}_{X'}^V(f)$ is naturally diffeomorphic to the manifold $\mathcal{G}_X\times_{\mathcal{G}^{X}_{X}}\mathcal{G}^X_{X'}(f)$, which is the quotient, under the free and proper action of $\mathcal{G}^{X}_{X}$, of the fibered product
$$
\{(\gamma;(x,\gamma'))\in \mathcal{G}_X\times \mathcal{G}^X_{X'}(f)| x= s(\gamma)\}.
$$
Recall that $((\gamma;(x,\gamma'))\sim (\gamma\alpha;\alpha^{-1}(x,\gamma'))$ for $\alpha\in \mathcal{G}_X^X$ with $r(\alpha)=s(\gamma)$. More precisely, let
$$
\phi_0: \{(\gamma;(x,\gamma'))\in \mathcal{G}_X\times \mathcal{G}^X_{X'}(f)| x= s(\gamma)\} \longrightarrow \mathcal{G}_{X'}^V(f)
$$
be defined as $\phi_0(\gamma;(x,\gamma'))=(r(\gamma),{f}(\gamma)\gamma')$. Then it is easy to check that $\phi_0$ induces a well defined map
$$
\phi : \mathcal{G}_X\times_{\mathcal{G}^{X}_{X}}\mathcal{G}^X_{X'}(f) \longrightarrow \mathcal{G}_{X'}^V(f)
$$
which  is a diffeomorphism.

Define $\pi_1: \mathcal{G}^V_{X'}(f)\rightarrow V$ and $\pi_2: \mathcal{G}^V_{X'}(f)\rightarrow \mathcal{G}'_{X'}$ by projecting onto the first and second factor, respectively. Then the hermitian bundle $E\to V$ allows to define the pre-Hilbert module $C_c^\infty (\cG_X, r^*E)$ over the  pre-$C^*$-algebras $C_c^\infty (\cG_X^X)$. The maximal completion of  $C_c^\infty (\cG_X, r^*E)$ will be denoted $\maE_{X, E}$ and it is a Hilbert module over the maximal $C^*$-algebra $C^*(\cG_X^X)$. We define similarly the Hilbert module $\maE_{X', E'}$ over $C^*(\cG_{X'}^{X'})$ associated with the hermitian bundle $E'\to V'$.

We  define a map $\Phi_f: C_c^\infty (\mathcal{G}^V_{X'}(f),\pi_1^\ast E)\rightarrow C_c^\infty (\mathcal{G}_X\times_{\mathcal{G}_X^X} \mathcal{G}_{X'}^X(f),(r\circ pr_1)^\ast E)$ as follows:
\begin{equation}
\label{Phif}\Phi_f(\xi)[\gamma;(s(\gamma),\gamma')]=\xi(r(\gamma),{f}(\gamma)\gamma'),\quad \text{ for }\xi \in C_c^\infty (\mathcal{G}^V_{X'}(f),\pi_1^\ast E).
\end{equation}
We thus have $ \Phi_f(\xi)[\gamma;(s(\gamma),\gamma')]\in E_{r(\gamma)}$. In the same way, we define the map
$$
\nu_f: C_c^\infty (\mathcal{G}_X,r^\ast E) \otimes C_c^\infty (\mathcal{G}_{X'}^X(f))\longrightarrow  C_c(\mathcal{G}_X\times_{\mathcal{G}^{X}_{X}}\mathcal{G}^X_{X'}(f),(r\circ \pi_1)^\ast E),
$$
by setting
$$
\nu_f (\xi\otimes \eta)[\gamma;(x,\gamma')]:= \sum_{\alpha\in \mathcal{G}^{s(\gamma)}_{X}} \xi(\gamma\alpha)\eta(\alpha^{-1}x',{f}(\alpha^{-1})\gamma').
$$

The spaces $C_c^\infty (\mathcal{G}^V_{X'}(f),\pi_1^\ast E)$ and $C_c^\infty (\mathcal{G}_X\times_{\mathcal{G}^{X}_{X}}\mathcal{G}^X_{X'}(f),(r\circ \pi_1)^\ast E)$ are naturally endowed with the structure of right $C_c^\infty ({\cG'}_{X'}^{X'})$-modules.
Using the diffeomorphism $\phi: \mathcal{G}_X\times_{\mathcal{G}_X^X}\cG_{X'}^X(f) \xrightarrow{\cong} \mathcal{G}^V_{X'}(f)$ described above we have $\Phi_f(\xi)= \xi\circ \phi$.
The inner product on $C_c^\infty (\mathcal{G}_X\times_{\mathcal{G}_X^X} \mathcal{G}_{X'}^X(f),(r\circ pr_1)^\ast E)$ is defined as
$$
<\xi_1,\xi_2>:= <\xi_1\circ \phi,\xi_2\circ \phi>,
$$
where the inner product on the RHS is the one defined on $C_c^\infty (\mathcal{G}^V_{X'}(f),\pi_1^\ast E)$ by the formula:
$$
<\xi,\eta>(\gamma'):= \int _{(v, \gamma'_1)\in V\times \maG'_{r(\gamma')}, f(v)=r(\gamma'_1)} <\xi(v,\gamma'_1),\eta(v,\gamma'_1\gamma)>_{E_v} d\alpha (v).
$$
Here, we assume for simplicity and since we shall only be interested in leafwise homotopy equivalences, that the inverse image of a leaf by $f$ is a finite union of leaves, and $d\alpha$ is the fixed Borel measure on the leaves of $(V,\maF)$. The general case introduces some tedious technicalities that we don't address here.
When $f$ is an oriented  leafwise homotopy equivalence, it is clear that the inverse image of a leaf is a leaf, and then the scalar product becomes:
$$
<\xi,\eta>(\gamma'):=  \int_{v \in L_{\gamma'}} \sum_{\gamma'_1\in \mathcal{G}'^{f(v)}_{r(\gamma')}} <\xi(v,\gamma'_1),\eta(v,\gamma'_1\gamma)>_{E_v} d\alpha (v).
$$
Here $L_{\gamma'}$ is the leaf in $V$ such that $f(L_{\gamma'})= L'_{r(\gamma')}$.

We denote the completion of $C_c^\infty (\mathcal{G}^V_{X'}(f),\pi_1^\ast E)$ with respect to the maximal $C^*$-norm, by   $\maE_{X', E}(f)$.  The completion of $C_c^\infty (\mathcal{G}_X\times_{\mathcal{G}^{X}_{X}}\mathcal{G}^X_{X'}(f),(r\circ \pi_1)^\ast E)$, again with respect to the maximal norm, is denoted by $\maE_{X,X';E}(f)$.

\begin{proposition}
\label{Iso1}
The above maps $\Phi_f$ and $\nu_f$ induce  isomorphisms of Hilbert modules over the $C^*$-algebra $C^*({\cG'}_{X'}^{X'})$:
$$
\Phi_f: \maE_{X', E}(f) \rightarrow \maE_{X,X';E}(f)\text{ and } \nu_f: \maE_{X, E} \otimes_{C^*(\mathcal{G}_X^X)} \maE_{X'}^X(f)\rightarrow \maE_{X,X';E}(f).
$$
\end{proposition}

\begin{proof}
The map $\Phi_f$ is clearly an isometry and  $\Phi_f$ is obviously surjective with  inverse given by the map induced by $\phi^{-1}$. Therefore $\Phi_f$ is an isometric isomorphism of Hilbert modules.

Also, we can follow the proof of Proposition \ref{composition} to deduce that $\nu_f$ extends to an isometric isomorphism of Hilbert modules. Notice that an isometric isomorphism is obviously adjointable with the adjoint given by the inverse.
\end{proof}

\begin{definition}
 The composition map $\nu_f^{-1}\circ \Phi_f$ will be denoted by $\ep_f$. So $\ep_f$ is an isomorphism of Hilbert modules over $C^*({\cG'}_{X'}^{X'})$:
$$
\ep_f:  \maE_{X', E}(f) \longrightarrow \maE_{X, E} \otimes_{C^*(\mathcal{G}_X^X)} \maE_{X'}^X(f).
$$
\end{definition}

\begin{proposition}
Let $\xi$ be an element of $C_c^\infty (\cG_{X'}^V(f), \pi_1^*E)$, then we have
$$
(\ep_f \circ \td^f )(\xi) = ((\td \otimes id)\circ \ep_f)(\xi),
$$
More precisely, $\Phi_f$ and $\nu_f$ are both chain maps.
\end{proposition}

\begin{proof}
The identification $\nu_f$ is clearly a chain map and it will be forgotten in this proof. Recall that we have a diffeomorphism
$$
\phi: \maG_X \times_{\maG_X^X} \maG_{X'}^X (f) \longrightarrow \maG_{X'}^V (f) \text{ given by } \phi ([\gamma, (x, \gamma')]) = (r(\gamma), f(\gamma)\gamma').
$$
Moreover, the map $r\circ \pi_1: \maG_X\times \maG_{X'}^X(f) \to V$ induces a smooth covering map $r_1: \maG_X\times_{\maG_X^X} \maG_{X'}^X(f) \to V$ such that $\pi_1 \circ \phi = r_1$. Here projections on the first factor are denoted $\pi_1$. The differential $\td^f$ is by definition the pull-back differential of the leafwise de Rham diffrential $d$ under the covering map $\pi_1: \maG_{X'}^V (f)\to V$, and it can be denoted by the suggestive notation $\pi_1^*d$. The differential $\td \otimes id$ is by definition the differential induced on the quotient manifold $\maG_X \times_{\maG_X^X} \maG_{X'}^X (f)$ by the pull-back under $r\circ \pi_1: \maG_X\times \maG_{X'}^X(f) \to V$ of the same leafwise de Rham differential $d$. So,
$$
\td^f = \pi_1^*d \text{ and } \td \otimes id \text{ is induced by } (r\circ \pi_1)^* d.
$$
Denote by $p: \maG_X\times \maG_{X'}^X(f) \to \maG_X \times_{\maG_X^X} \maG_{X'}^X (f)$ the covering map. If $\eta$ is a smooth leafwise form on $V$ then we can write
$$
r_1^* (d\eta) = (\td\otimes id) (r_1^*\eta) \text{ and } \pi_1^* (d\eta) = \td^f (\pi_1^*\eta).
$$
Therefore,
\begin{eqnarray*}
 (\Phi_f \circ \td^f) (\pi_1^*\eta) & = & (\phi^* \circ \pi_1^*) d\eta\\
&=& (\pi_1\circ \phi)^* \phi^* d\eta\\
& = & r_1^* d\eta\\
& = & (\td\otimes id) (r_1^*\eta)\\
& = & (\td\otimes id) \phi^*(\pi_1^*\eta)
\end{eqnarray*}
So,  $\Phi_f \circ \td^f = (\td\otimes id) \circ \Phi_f$ on sections of the form $\pi_1^*\eta$. Now, the statement being local, a classical argument again allows to deduce the allowed relation for any smooth form on $\maG_{X'}^V(f)$.
\end{proof}

{\bf{We assume from now on that $E:=\Lambda^\ast T^\ast_{\C} \mathcal{F}$ and $E':=\Lambda^\ast T^\ast_{\C} \mathcal{F}'$ are the leafwise grassmannian bundles with their hermitian structures inherited from the leafwise metrics.}}

We proceed first to define a smooth pull-back map induced by $f$ between the corresponding spaces of smooth differential forms. Let $d'$ denote the longitudinal de Rham differential along the leaves of $(V', \maF')$, and let $\tilde{d'}=(\tilde{d'}_{x'})_{x'\in V'}$ be its lift under the covering map $r$, to the fibers of the monodromy groupoid $\cG'$. Similarly, let $d$ be the de Rham differential on the leaves of $(V,\maF)$ and $\tilde{d}$ its lift by $r$ the the fibers of $\cG$. We also set $\td^f =(\td^f_{x'})_{x'\in X'}$ where $\td^f_{x'}$ is the lift of $d$, by the first projection $\pi_1$, to $\mathcal{G}_{x'}^V(f):=\{(v,\gamma')\in V\times \mathcal{G}_{x'}| r(\gamma')=f(v)\}$.
\begin{definition}
 For $\omega'\in C_c^\infty(\mathcal{G}'_{X'},r^\ast E')$, we define
$$
\Psi_f(\omega')(v,\gamma')=(^t{f_{\ast_{v}}})[(\pi_2^{!} \omega')(v,\gamma')], \quad  (v,\gamma')\in \mathcal{G}^V_{X'}(f).
$$
where $^t{f_{\ast_{v}}}: \Lambda^\ast  T^\ast_{f(v)} \mathcal{F}' \rightarrow  \Lambda^\ast T^\ast_{v}\mathcal{F}$ is the transpose of the tangent map $f_{\ast,v}: T_{v}\mathcal{F}\rightarrow T_{f(v)}\mathcal{F}'$ and $\pi_2^!$ is the pullback via $\pi_2$ of elements of $C_c^\infty(\mathcal{G}'_{X'},r^\ast E')$, so  $(\pi_2^{!} \omega')(v,\gamma')=\omega'(\gamma')\in E'_{r(\gamma')=f(v)}$.
\end{definition}
When $f$ is a leafwise homotopy equivalence, it is uniformly proper on the different foliated spaces, see \cite{HeitschLazarov, BenameurHeitschJDG}, therefore in this case the support of $\Psi_f (\omega')$ is also compact and we get in this way a map
$$
\Psi_f:C_c^\infty(\mathcal{G}'_{X'},r^\ast E')\longrightarrow C_c^\infty(\mathcal{G}^V_{X'}(f),\pi_1^\ast E).
$$
We also denote by $\Psi_f$ the same map acting on smooth, not necessarily compactly supported, sections. Denote by $f^*\alpha'$ the usual pull-back  by $f$ of a differential form $\alpha'$ on $V'$.

\begin{proposition}
We have the following properties:
\begin{enumerate}
 \item For any $\alpha'\in C_c^\infty(L' ,E')$, $ \pi_1^!(f^\ast \alpha')= \Psi_f(r^!\alpha')$.
\item $\tilde{d}^f_{x'}\circ \Psi_f=\Psi_f\circ \tilde{d}'_{x'}$ on $C_c^\infty(\mathcal{G}'_{x'}, r^*E')$.
\end{enumerate}
where we have denoted, as for $\pi_2$ above, pullbacks via $\pi_1$ and $r$ by $\pi_1^!$ and $r^!$ respectively.
\end{proposition}

\begin{proof}
Notice that for $(v,\gamma')\in \mathcal{G}_{X'}^V(f)$, we have
$$
f\circ\pi_1(v,\gamma')=f(v)=r(\gamma')= r\circ \pi_2(v,\gamma').
$$
Therefore, we compute
\begin{eqnarray*}
\pi_1^!(f^\ast\alpha')(v,\gamma')&=&(f^\ast\alpha')(v)\\
&=& ^t{f_{\ast_{v}}}(\alpha'_{f(v)})\\
&=& ^t{f_{\ast_{v}}} [\alpha'_{(f\circ \pi_1)(v,\gamma')}] \\
&=& ^t{f_{\ast_{v}}} [(f\circ \pi_1)^!(\alpha'_{(v,\gamma')})] \\
&=& ^t{f_{\ast_{v}}} [(r\circ \pi_2)^!(\alpha'_{(v,\gamma')})] \hspace{0.2cm} \\
&=& ^t{f_{\ast_{v}}} [\pi_2^!\circ r^!(\alpha'_{(v,\gamma')})] \\
&=& \Psi_f(r^!\alpha')(v,\gamma')
\end{eqnarray*}
hence the first item.

For the second item, we fix $s\in C_c^\infty(\mathcal{G}_{x'},r^\ast E')$. Since the statemnt is local in the leaf $L'_{x'}$, we can assume that our section $s$  can be written in the  form:
$$
s= \sum_{i}h_i r^!\alpha'_i, \text{ where }h_i\in C^\infty_c(\mathcal{G}'_{x'})\text{ and } \alpha'_i\in C_c^\infty(L'_{x'},E').
$$
Now, notice that for any $\alpha'\in C^\infty (L'_{x'}, E')$, the following relation holds
$$
(\tilde{d}_{x'}^f\circ \Psi_f) (r^!\alpha') = (\Psi_f \circ \tilde{d}_{x'}) (r^!\alpha').
$$
Indeed, since $\Psi_f(r^!\alpha') = \pi_1^! (f^*\alpha')$ and $\tilde{d}_{x'}^f$ is precisely the pull-back operator of $d_x$ under $\pi_1$, this result is a consequence of the fact that $(f^* \circ d'_{x'}) (\alpha')= (d_x\circ f^*)(\alpha')$.
The differential  $\tilde{d}_{x'}^f$ can also be described as the de Rham differential on the manifold $\cG_{x'}^V(f)$ since this latter is the total space of a covering over $L_x$ given precisely by the projection $\pi_1$. So we compute
\begin{eqnarray}
\label{gen1}
(\tilde{d}^f_{x'}\circ \Psi_f)(\sum_i h_i r^!\alpha'_i)&=& \tilde{d}_{x'}^f\left[ \sum_i (\pi_2^!h_i) (\Psi_f r^!\alpha')\right] \nonumber\\
&=& \sum_i \left(\tilde{d}^f_{x'} \pi_2^!h_i\wedge \Psi_fr^!\alpha'_i+ \pi_2^!h_i \tilde{d}_{x'}^f\Psi_f r^!\alpha'_i\right) \nonumber\\
&=& \sum_i \left(\tilde{d}^f_{x'} \pi_2^!h_i\wedge \Psi_fr^!\alpha'_i+ \pi_2^!h_i \Psi_f (\tilde{d}_{x'} r^!\alpha'_i)\right)
\end{eqnarray}
We also have,
\begin{eqnarray}\label{gen2}
(\Psi_f\circ \tilde{d}_{x'})(\sum_i h_i r^!\alpha'_i)&=& \Psi_f\left[\sum_i \tilde{d}_{x'}h_i\wedge r^!\alpha'_i+ h_i \tilde{d}_{x'}r^!\alpha'_i\right] \nonumber\\
&=& \sum_i \left(\Psi_f(\tilde{d}_{x'}h_i\wedge r^!\alpha'_i)+ \Psi_f(h_i \tilde{d}_{x'}r^!\alpha'_i)\right)
\end{eqnarray}
It thus remains to show that for a given smooth function $h$ on $\cG'_{x'}$, the following relation holds
$$
(\Psi_f\circ \tilde{d}_{x'}) (h) = (\td_{x'}^f \circ \pi_2^!) (h).
$$
Since $\tilde{d}_{x'}$ is a differential operator, this is again a local statement and we can use the covering $\cG'_{x'}\to L'_{x'}$ to reduce to an open submanifold $\tU$ of $\cG'_{x'}$ which is diffeomorphic through $r$ to an open  submanifold $U$ of $L'_{x'}$. Therefore, we can suppose that $h$ is the pull-back $r^!h_0=r^*h_0$ of a smooth function $h_0$ on $U$. But then
$$
(\Psi_f\circ \tilde{d}_{x'}) (h) = \Psi_f (r^* d_{x'}h_0) \text{ and } (\td_{x'}^f \circ \pi_2^!) (h)= \pi_1^! (d_x f^*h_0)= (\pi_1^!\circ f^*) (d_xh_0).
$$
Evaluating at $(v,\gamma')\in \cG_{x'}^V(f)$ we thus get
$$
(\Psi_f\circ \tilde{d}_{x'}) (h) (v,\gamma')= ^tf_{*,v} (d_x h_0)_{r(\gamma')}\text{ while } (\td_{x'}^f \circ \pi_2^!) (h) (v,\gamma')=^tf_{*,v}(d_xh_0)_{f(v)}.
$$
Since $f(v)=r(\gamma')$, the proof is finished by using \eqref{gen1} and \ref{gen2}.
\end{proof}

\begin{remark}
The map $\Psi_f$ is  not bounded in general (and hence not adjointable). In fact, it is even not regular in general!
\end{remark}

We denote by $\Delta$ and $\Delta'$ the Laplace operator along the leaves of the monodromy groupoids $\maG$ and $\maG'$ respectively. So, $\Delta = (\td + \td^*)^2$ where $\td^*$ is the formal adjoint and $\Delta = (\Delta_v)_{v\in V}$ where $\Delta_v$ is the Laplace operator on differential forms of $\cG_x$. Therefore, and since $\Delta$ is a differential operator, it yields a linear map
$$
\Delta : C_c^\infty (\cG_X, r^*E) \longrightarrow  C_c^\infty (\cG_X, r^*E).
$$

\begin{lemma}\cite{Vassout}
The operator $\Delta$ is $C_c^\infty(\cG_X^X)$-linear and extends to a regular self-adjoint operator, still denoted $\Delta$, on the Hilbert $C^*(\cG_X^X)$-module $\maE_{X,E}$. The similar statement holds of course for $\Delta'$.
\end{lemma}

Using the continuous functional calculus theorem for regular self-adjoint operators, we define for any continuous bounded function $\phi$ on $\R$, a
bounded operator $\phi(\Delta)$ on the Hilbert $C^*(\cG_X^X)$-module $\maE_{X,E}$.

\begin{definition}
 Let $\phi$ be a function on $\R$ which is the Fourier transform of an element of $ C_c^\infty(\R)$ and such that $\phi(0)=1$. We define
$$
\Psi_f^\phi:= \Psi_f\circ \phi(\Delta') : C_c^\infty ({\cG'}_{X'}, r^*E') \longrightarrow  C^\infty (\cG_{X'}^V(f), \pi_1^*E).
$$
\end{definition}

\begin{proposition}
Assume that $f$ is uniformly proper \cite{BenameurHeitschLHE}, then the $C_c^\infty(\cG_X^X)$-linear map $\Psi_f^\phi$ extends to an adjointable operator
$$
\Psi_f^\phi : \maE_{X',E'} \longrightarrow  \maE_{X', E}(f).
$$
\end{proposition}

\begin{proof}\
 To see that $\Psi_f^\phi:= \Psi_f\circ \phi(\Delta')$ extends to an adjointable operator on Hilbert modules, we compute its Schwartz kernel. Let $k_\phi$ denote the Schwartz kernel of $\phi(\Delta')$. Since $\phi$ is chosen such that its Fourier transform is smooth compactly supported, the kernel $k_\phi$ is a smooth section with compact support over $\mathcal{G}'$ \cite{Roe}. We set
$$
K_f^\phi(v,\gamma') = ^t f_{\ast,v}\circ k_\phi(\gamma') \quad \in Hom(E'_{s(\gamma')}, E_v).
$$
Then for any $\omega'\in C_c^\infty (\cG'_{X'}, r^*E')$, we have
$$
\Psi_f^\phi (\omega') (v, \gamma') = \int_{\cG'_{s(\gamma')}} K_f^\phi(v,\gamma'{\gamma'}_1^{-1}) \omega'({\gamma'}_1^{-1}) d\lambda_{s(\gamma')} (\gamma'_1).
$$
Here  $d\lambda_{x'}$ is the $\cG'$-invariant Haar system pulled back from the Borel measure $\alpha'$ on the leaves.
 Since $k_\phi$ is smooth with compact support in $\mathcal{G}'$,   $K_\phi^f$ also is smooth with compact support in $\mathcal{G}^V_{X'}(f)$. Then classical arguments show that $\Psi^f_\phi$ extends to an adjointable operator as claimed.
\end{proof}
Recall from \cite{BenameurHeitschLHE} that a leafwise homotopy equivalence is always uniformly proper.
We are now in position to define a pull-back map associated with the leafwise oriented leafwise smooth map $f: (V,\maF) \rightarrow (V', \maF')$  (satisfying \ref{Etale}) and with respect to $\phi$ and to the leafwise metric defining $\Delta'$.

\begin{definition}
Let $\phi$ be the Fourier transform of an element of $ C_c^\infty(\R)$ such that $\phi(0)=1$. Let as before $X$ and $X'$ be complete transversals for $(V,\maF)$ and $(V',\maF')$ respectively. Then the pull-back map by $f$ associated with $\phi$ is the adjointable $C^*({\cG'}_{X'}^{X'})$-linear operator
$$
f_\phi^* : = \ep_f \circ \Psi_f^\phi =\nu_f^{-1}\circ  \Phi_f \circ \Psi_f \circ \phi(\Delta'): \maE_{X', E'} \longrightarrow \maE_{X,E}\otimes_{C^*(\cG_X^X)} \maE_{X'}^X(f).
$$
\end{definition}

\begin{lemma}\label{Independant}
\begin{enumerate}
 \item For any function $\phi$ which is the Fourier transform of an element of $C_c^\infty (\R)$ and which satisfies $\phi(0)=1$, the operator $\phi(\Delta): \maE_{X, E} \to \maE_{X,E}$ is an adjointable chain map which induces the identity on cohomology.
\item The two adjointable chain maps $f_\phi^*$ and $(\phi(\Delta)\otimes id)\circ \ep_f\circ \Psi_f$ induce the same map on cohomologies.
\end{enumerate}
\end{lemma}

\begin{proof}
\begin{enumerate}
\item
 As $\phi$ has compactly supported Fourier transform, it is easy to check that $\Im(\phi(\Delta_X))\subseteq Dom( d_X)$. Furthermore, since the Fourier
transform of a compactly supported smooth function is an entire function whose restriction to $\R$ is Schwartz, we get that $\phi$ is entire. Then,
following the arguments in  \cite{HeitschLazarov}, we consider the holomorphic functional calculus for the self-adjoint regular operator $\Delta_X$,
which makes sense as the resolvent map $z \mapsto (zI-\Delta_X)^{-1}$ is analytic on the resolvent of $\Delta_X$ in $\C$ (cf. Result 5.23 in
\cite{Kustermans}). Therefore, choosing a curve $\gamma$ in $\C$ that does not intersect $\R^+$ and surrounds it, as in  \cite{HeitschLazarov}, one can write
$$
 \phi(\Delta_X)= \frac{1}{2\pi i} \int_{\gamma} \phi(z) (zI-\Delta_X)^{-1} dz .
 $$
Now, for $z \in \C$ in the resolvent of $\Delta_X$, we have $(zI-\Delta_X)^{-1} d_X =d_X(zI-\Delta_X)^{-1}$ and thus $\phi(\Delta_X)d_X= d_X\phi(\Delta_X)$, more precisely, the image of the adjointable operator $\phi(\Delta_X)$ is contained in the domain of $d_X$ and the two operators  $\phi(\Delta_X)d_X$ and $ d_X\phi(\Delta_X)$ extend to adjointable operators which coincide on the domain of $d_X$ and hence coincide. Similar arguments show that $\phi(\Delta_X)\delta_X= \delta_X\phi(\Delta_X)$.

  Now to show that $\phi(\Delta_X)$ induces the identity map on   cohomology we proceed as follows. As $\phi$ is entire with $\phi(0)=1$, the function $\psi$ given by $\psi(x)= \frac{\phi(x)-1}{x}$ is also entire and in particular smooth on $\R$. Now, there exists a sequence of Schwartz functions with compactly supported Fourier transforms $(\alpha_n)_{n\in \N}$ such that $\alpha_n \xrightarrow{n\rightarrow \infty} \psi$ in the $\|\bullet \|_{\infty}$ norm. Consequently, if $v\in \mathcal{E}^k_c$ for $k= 0,1,2,...,p$, we get
$$
\alpha_n(\Delta_X)v\xrightarrow{n\rightarrow \infty} \psi(\Delta_X)v \text{ and } d_X(\alpha_n(\Delta_X)v) = \alpha_n(\Delta_X)(d_Xv)\xrightarrow{n\rightarrow \infty} \psi(\Delta_X)(d_Xv).
$$
Therefore $\psi(\Delta_X)v \in Dom(d_X)$ and $\psi(\Delta_X)d_X= d_X\psi(\Delta_X)$ on $\mathcal{E}^k_c$.

Now let $\omega\in \Ker(d_X)$. Then there exists a sequence $(\omega_n)_{n\geq 0}$ such that each $\omega_n$ is a compactly supported smooth form, $\omega_n \xrightarrow{n\rightarrow \infty} \omega$, and $d_X\omega_n\rightarrow 0$ in the Hilbert modules. We then have,
\begin{eqnarray}
\label{phidelta}
   &&\phi(\Delta_X)\omega_n -\omega_n= \psi(\Delta_X)\Delta_X(\omega_n) = \\
   &=& \psi(\Delta_X)(d_X\delta_X \omega_n) + \psi(\Delta_X)(\delta_X d_X \omega_n)   \text { (Notice $\Im (d_X|_{\mathcal{E}^k_c}) \subseteq \mathcal{E}^{k+1}_c, \Im(\delta_X|_{\mathcal{E}^k_c}) \subseteq \mathcal{E}^{k-1}_c$)}\nonumber\\
   &=& d(\psi(\Delta_X )\delta_X \omega_n) + \psi(\Delta_X)\delta_X (d_X\omega_n)   \nonumber\text{   (by (ii) above) }
\end{eqnarray}
But, on compactly supported smooth forms, we have
\begin{eqnarray*}
\psi(\Delta_X)\circ \delta_X &=& \psi(\Delta_X) [(I+\Delta_X)(I+\Delta_X)^{-1}] \delta_X \\
 &=& [\psi(\Delta_X) (I+\Delta_X)][(I+\Delta_X)^{-1} \delta_X ] \text{     (since $\Im(I+\Delta_X)^{-1}\subseteq Dom(I+\Delta_X)$) } \\
 &=& [\psi(\Delta_X)+\phi(\Delta_X)-I] [(I+\Delta_X)^{-1} \delta_X ]
\end{eqnarray*}
Now $\psi(\Delta_X)+\phi(\Delta_X)-I$ is clearly adjointable as $\phi$ and $\psi $ are bounded smooth functions and  $(I+\Delta_X)^{-1} \delta_X $ is adjointable because it is a pseudo-differential operator of negative order \cite{Vassout}. Hence $\psi(\Delta_X)\circ \delta_X$ is an adjointable operator.

We thus get
$$
\psi(\Delta_X)\delta_X (\omega_n)\xrightarrow{n\rightarrow \infty} \psi(\Delta_X)(\delta_X\omega)  \text{ and } \psi(\Delta_X)\delta_X (d\omega_n)\xrightarrow{n\rightarrow \infty} 0.
 $$
Hence by $\ref{phidelta}$, we get
$$
 \psi(\Delta_X )\delta_X (\omega_n)\xrightarrow{n\rightarrow \infty} \psi(\Delta_X)(\delta_X\omega),
 $$
and
$$
d(\psi(\Delta_X )\delta_X \omega_n)) =  \phi(\Delta_X)\omega_n -\omega_n - \psi(\Delta_X)\delta_X (d_X\omega_n) \xrightarrow{n\rightarrow \infty} \phi(\Delta)\omega-\omega$$
Thus the above two limits together imply
$$
\psi(\Delta_X) \delta_X \omega\in Dom(d_X)\text{ and }\phi(\Delta_X)\omega -\omega = d_X(\psi(\Delta)\delta \omega) \subseteq \Im(d_X).
$$
So $\phi(\Delta_X)\omega -\omega=0$ on   cohomology and $\phi(\Delta_X)$ is the identity map on   cohomology.
\item We may compose the adjointable chain map $f_\phi^*$ on the left by the chain map $\phi(\Delta)\otimes id$ and get an adjointable chain map which induces, by the first item, the same map as $f_\phi^*$ on cohomlogies. In the same way, we can compose $(\phi(\Delta)\otimes id)\circ \ep_f\circ \Psi_f$ on the right by the chain map $\phi (\Delta')$ and get an adjointable chain map which induces, by the first item again, the same map as $(\phi(\Delta)\otimes id)\circ \ep_f\circ \Psi_f$ on cohomologies. This completes the proof of the second item.
\end{enumerate}
\end{proof}

We summarize our results in the following proposition.

\begin{theorem}\label{pullback}
The pull-back map $f_\phi^*$ is an adjointable operator  which is a chain map between the de Rham Hilbert-Poincar\'e complexes. Moreover, The map induced by $f_\phi^*$ on cohomology does not depend on the choices of $\phi$ and of the leafwise metric on $(V', \maF')$.
\end{theorem}

\begin{proof}
Only the last part of the statement needs to be proved.
From the second item of Lemma \ref{Independant}, we see that the map induced by $f_\phi^*$ does not depend on the leafwise metric on $(V',\maF')$.
Now, assume that $\psi$ is another function which is the Fourier transform of an element of $C_c^\infty(\R)$ and which satisfies $\psi (0)=1$.
Then $\phi\psi$ satisfies the same conditions as $\phi$ and $\psi$ and we have
$$
f_\phi^* \circ \psi (\Delta') = f_\psi^* \circ \phi (\Delta').
$$
By the first item of Lemma \ref{Independant}, $\psi (\Delta')$  and $\phi (\Delta')$ are  chain maps which induce the identity on cohomologies, therefore  $f_\phi^*$ and  $f_\psi^*$ induce the same map on cohomologies.
\end{proof}

\subsection{Functoriality of the pull-back}

Recall that $f:(V,\maF)\to (V', \maF')$ is a leafwise oriented smooth map which satisfies Assumption \ref{Etale}.
 We have fixed complete transversals $X$ and $X'$ for the foliations $(V, \maF)$ and $(V', \maF')$ respectively.
Let $g: (V',\maF')\to (V'', \maF'')$ be another leafwise oriented smooth map, also satisfying \ref{Etale}, and $X''$ a complete transversal in $(V'', \maF'')$. Inorder to work with adjointable operators, we shall also assume that $f$ and $g$ are uniformly proper maps between the two foliations. We defined the pull-back adjointable operators associated with a fixed nice cutoff function $\phi$:
$$
f_\phi^*: \maE_{X', E'} \rightarrow \maE_{X,E}\otimes \maE_{X'}^X(f) \text{ and } g_\phi^*: \maE_{X'', E''} \rightarrow \maE_{X',E'}\otimes \maE_{X''}^{X'}(g).
$$
We also proved that $f_\phi^*$ and $g_\phi^*$ are chain maps between the de Rham Hilbert-Poincar\'e complexes. The goal of this subsection is to prove the expected functoriality property.

We denote by $\Xi_{g,f}^{X'}$ the isomorphism
$$
\Xi_{g,f}^{X'}:  \maE_{X''}^{X} (g\circ f) \longrightarrow   \maE_{X'}^X(f) \otimes_{C^*({\cG'}_{X'}^{X'})} \maE_{X''}^{X'}(g),
$$
described in Proposition \ref{composition}. Notice that this is a bimodule isomorphism. Then we can state:
\begin{theorem}\label{Functoriality}
The following diagram of adjointable chain maps induces a commutative diagram between the corresponding $C^*({\cG'}_{X'}^{X'})$ Hilbert-Poincar\'e de Rham cohomologies

\medskip
\begin{picture}(415,80)
\put(105,60){$\maE_{X'', E''}$}
\put(120,50){ $\vector(0,-1){25}$}
\put(45,10){$\maE_{X, E}\otimes_{C^*({\cG}_{X}^{X})} \maE_{X''}^{X} (g\circ f)$}
\put(75,38){$(g\circ f)_\phi^*$}
\put(183,70){$g_\phi^*$}
\put(160,64){\vector(1,0){70}}
\put(160,13){\vector(1,0){70}}
\put(163,25){$id \otimes_{C^*({\cG}_{X}^{X})} \Xi^{X}_{g,f}$}
\put(240,60){$\maE_{X',E'}\otimes_{C^*({\cG'}_{X'}^{X'})}\maE_{X''}^{X'}(g)$}
\put(260,50){ $\vector(0,-1){25}$}
\put(230,10){$\maE_{X, E}\otimes_{C^*({\cG}_{X}^{X})} \maE_{X'}^X(f) \otimes_{C^*({\cG'}_{X'}^{X'})} \maE_{X''}^{X'}(g).$}
\put(270,38){$f_\phi^*\otimes_{C^*({\cG'}_{X'}^{X'})} id $}
\end{picture}
\vspace{-2cm}
\begin{Equation}\label{CommutativeDiagram}
\end{Equation}
\vspace{1.5cm}\noindent
\end{theorem}

We devote the rest of this subsection to the proof of this theorem.
Indeed, we shall prove more precisely that the map $((\ep_g \circ \Psi_g)\otimes_{C_c^\infty(\cG_X^X)} id)\circ (\ep_f \circ \Psi_f)$ coincides
exactly with $\ep_{f\circ g} \circ \Psi_{f\circ g}$ on smooth compactly supported sections, when we identify
{$\maE_{X'}^{X} (f) \otimes \maE_{X''}^{X'} (g)$ with $\maE_{X''}^{X} (g\circ f)$ }. For simplicity, we shall denote for an operator $T$ and for an algebra $A$ by $T\otimes id$ the expression $T\otimes_{A} id$, when the algebra $A$ is clearly understood and no confusion can occur. Even if the maps $\Psi_f$ and $\Psi_g$ are not adjointable, we shall restrict to smooth compactly supported sections and by using eventually the regularization $\phi(\Delta')$ or $\phi(\Delta)$, we shall easily deduce the corresponding results at the level of Hilbert module completions.

{
\begin{proposition}\label{Lambda}
There exists an isometric isomorphism of Hilbert modules
$$
\maE_{X', E}^{V}(f) \otimes_{C^*({\cG'}_{X'}^{X'})} \maE_{X''}^{X'} (g) \longrightarrow \maE_{X, E''}^{V''} (g\circ f).
$$
\end{proposition}
}

\begin{proof}
The proof is a straightforward generalization of the easier proof corresponding to the case where the hermitian bundles  $E$ and $E''$ are the
trivial line bundles. In this latter case, we give below  the proof of Proposition \ref{Iso} which adapts immediately to the replacement of $X$ by $V$,
and we thus leave the details of the precise modifications as an exercise.
\end{proof}

Denote by
{
$$
\Pi_{g,f}: \maE_{X,E''}^{V''} (g\circ f) \rightarrow \maE_{X', E}^{V}(f) \otimes_{C^*({\cG'}_{X'}^{X'})} \maE_{X''}^{X'} (g),
$$
}
the isomorphism described in Proposition \ref{Lambda}. Recall that $\Xi_{g,f}$ is the extension of a $C_c^\infty ({\cG'}_{X'}^{X'})$-linear map, still denoted $\Xi_{g,f}$, between the smooth compactly supported sections.
\begin{lemma}\label{comdiag1}
The following relation holds on $C_c^\infty (\cG'_{X'}, r^*E')$,
{
$$
\Pi_{g,f} \circ \Psi_{g\circ f}  = (\Psi_f\otimes id)\circ \ep_g\circ \Psi_g.
$$
}
\end{lemma}

\begin{proof}\
First consider the diffeomorphism (the proof of this is analogous to that of Proposition \ref{composition}),
{
$$
\lambda: \mathcal{G}_{X'}^{V}(f)\times_{\mathcal{G}^{X'}_{X'}} \mathcal{G}^{X'}_{X''}(g)\rightarrow \mathcal{G}^{V}_{X''}(g\circ f)\text{ given by }
\lambda[(v,\gamma');(s(\gamma'),\gamma'')]=(v,{g}(\gamma')\gamma'').
$$
}
Then $\lambda$ induces the map
{
$$
\Lambda_{g,f}: C_c(\mathcal{G}^{V}_{X''}(g\circ f), \pi^\ast_1 E)\rightarrow  C_c(\mathcal{G}_{X'}^{V}(f)\times_{\mathcal{G}^{X'}_{X'}} \mathcal{G}^{X'}_{X''}(g), (\pi_1\circ pr_1)^\ast E).
$$
}

By arguing as in the proof of Proposition \ref{composition}, $\Lambda_{g,f}$ is an isometry which extends to an isometric isomorphism between the Hilbert modules. Similar arguments allow to construct, as for $\nu_f$ in the previous paragraph, an isometric isomorphism $\mu$ which restricts to
{
$$
\mu_{g,f}: C_c^\infty (\mathcal{G}_{X'}^{V}(f)\times_{\mathcal{G}^{X'}_{X'}} \mathcal{G}^{X'}_{X''}(g) , (\pi_1\circ pr_1)^*E) \longrightarrow C_c^\infty ({\cG}_{X'}^{V}(f), \pi_1^*E) \otimes_{C_c^\infty (\cG_{X'}^{X'})} C_c^\infty (\cG_{X''}^{X'} (g)).
$$
}
A direct inspection shows that
$$
\mu_{g,f}\circ \Lambda_{g,f} = \Pi_{g,f}.
$$
{
On the other hand, if $\alpha\in C_c^\infty (\mathcal{G}_{X'}\times_{\mathcal{G}_{X'}^{X'} } \mathcal{G}_{X''}^{X'}(g)),(r\circ pr_1)^\ast E')$,  then we set
$$
{\widehat\Psi_f} (\alpha)[(v,\gamma');(s(\gamma'),\gamma'')]= ^t{f_{\ast,v}}(\alpha[\gamma';(s(\gamma'),\gamma'')])
$$
}
and define in this way
{
$$
{\widehat\Psi_f} : C_c^\infty (\mathcal{G}_{X'}\times_{\mathcal{G}_{X'}^{X'}}\mathcal{G}_{X''}^{X'}(g),(r\circ \pi_1)^\ast E') \longrightarrow  C_c^\infty (\mathcal{G}_{X'}^{V}(f)\times_{\mathcal{G}^{X'}_{X'}} \mathcal{G}^{X'}_{X''}(g) , (\pi_1\circ pr_1)^*E)
$$
}
which corresponds to $\Psi_f\otimes id$ through the isomorphisms. More precisely, the following diagram commutes
{
\medskip
\begin{picture}(415,80)
\put(35,60){$C_c^\infty (\mathcal{G}_{X'}\times_{\mathcal{G}_{X'}^{X'}}\mathcal{G}_{X''}^{X'}(g),(r\circ \pi_1)^\ast E')$}
\put(115,50){ $\vector(0,-1){25}$}
\put(15,10){$C_c^\infty (\cG_{X'}, r^*E')\otimes_{C_c^\infty (\cG_{X'}^{X'})} C_c^\infty (\cG_{X''}^{X'}(g))$}
\put(85,38){$(\nu)_g^{-1}$}
\put(200,70){$\widehat\Psi_f$}
\put(180,64){\vector(1,0){50}}
\put(180,13){\vector(1,0){50}}
\put(180,25){$\Psi_f \otimes id$}
\put(240,60){$C_c^\infty (\mathcal{G}^{V}_{X'}(f)\times_{\mathcal{G}^{X'}_{X'}}\mathcal{G}^{X'}_{X''}(g) , (\pi_1\circ pr_1)^*E)$}
\put(280,50){ $\vector(0,-1){25}$}
\put(240,10){$C_c^\infty ({\cG}_{X'}^{V}(f), \pi_1^*E) \otimes_{C_c^\infty (\cG_{X'}^{X'})} C_c^\infty (\cG_{X''}^{X'} (g)).$}
\put(285,38){$\mu_{g,f} $}
\end{picture}
\vspace{-2cm}
\begin{Equation}
\end{Equation}
}
\vspace{1.5cm}\noindent
Thus it remains to show that, on smooth compactly supported forms, we have
{
$$
\Lambda_{g,f}\circ \Psi_{g\circ f} = \hat\Psi_f \circ \Phi_g\circ \Psi_g.
$$
}
Now, for $\alpha\in C_c^\infty (\mathcal{G}^{V}_{X''}(g\circ f),\pi^\ast_1 E)$, we have
{
$$
(\Lambda_{g,f} \alpha)[(v,\gamma');(s(\gamma'),\gamma'')]:= \alpha(v,{g}(\gamma')\gamma'') \quad \in E_{v}.
$$
}

Let then $\beta''\in C_c(\mathcal{G}_{X''},r^\ast E'')$. Then we have
{
\begin{eqnarray*}\label{lhs6}
(\hat\Psi_f\circ \Phi_g \circ \Psi_g)(\beta'')[(v,\gamma');(s(\gamma'),\gamma'')]&=& ^t{f_{\ast,{v}}}[\Psi_g(\beta'')(r(\gamma'),{g}(\gamma')\circ \gamma'')]\nonumber\\
&=& ^t{f_{\ast,{v}}}[^t{g}_{\ast,r(\gamma')}(\beta''({g}(\gamma')\gamma''))]\nonumber\\
&=&^t ({g_{\ast,f(v)}\circ f_{\ast,v}})(\beta''({g}(\gamma')\gamma''))\nonumber\\
&=& (^t{(g\circ f)_{\ast,v}})(\beta''({g}(\gamma')\gamma''))
\end{eqnarray*}
}
{
On the other hand,
\begin{eqnarray*}\label{rhs6}
(\Lambda_{g,f} \circ \Psi_{g\circ f})(\beta'')[(v,\gamma');(s(\gamma'),\gamma')]&=& \Psi_{g\circ f}(v,{g}(\gamma')\gamma'')\nonumber\\
&=& (^t{(g\circ f)_{\ast,v}})(\beta''({g}(\gamma')\gamma''))
\end{eqnarray*}
}
which completes the proof of the lemma.
\end{proof}

Next we prove the following

\begin{lemma}\label{comdiag2}

With the above notations, the following diagram commutes

{
\medskip
\begin{picture}(415,80)
\put(60,60){$C_c^\infty ({\mathcal{G}}_{X''}^{V} (g\circ f),\pi_1^\ast E)$}
\put(115,50){ $\vector(0,-1){25}$}
\put(30,10){$C_c^\infty (\cG_{X}, r^*E)\otimes C_c^\infty (\cG_{X''}^{X}(g\circ f))$}
\put(85,38){$\ep_{g\circ f}$}
\put(195,70){$\Pi_{g,f}$}
\put(190,64){\vector(1,0){30}}
\put(190,13){\vector(1,0){30}}
\put(185,25){$id \otimes \Xi_{g,f}^{X'}$}
\put(240,60){$C_c^\infty (\mathcal{G}^{V}_{X'}(f), \pi_1^*E) \otimes C_c^\infty (\mathcal{G}^{X'}_{X''}(g))$}
\put(280,50){ $\vector(0,-1){25}$}
\put(240,10){$C_c^\infty ({\cG}_{X}, r^*E) \otimes C_c^\infty (\cG_{X'}^{X} (f)) \otimes C_c^\infty (\cG_{X''}^{X'}(g)).$}
\put(300,38){$\ep_f\otimes  id $}
\end{picture}
}
\vspace{-2cm}
\begin{Equation}
\end{Equation}
\vspace{1.5cm}\noindent
\end{lemma}

\begin{proof}\
We see the elements of {$C_c^\infty ({\cG}_{X}, r^*E) \otimes C_c^\infty (\cG_{X'}^{X} (f)) \otimes C_c^\infty (\cG_{X''}^{X'}(g))$} as sections over
{
$$
\cG_{X} \times_{{\cG}_{X}^{X}}\cG_{X'}^{X} (f) \times_{\cG_{X'}^{X'}} \cG_{X''}^{X''}(g).
$$}
Then a straightforward computation gives for $\alpha\in C_c^\infty (\mathcal{G}^{V}_{X''},\pi^\ast_1 E)$
{
\begin{eqnarray}\label{lhs7}
((\ep_f\otimes id)\circ\Pi_{g,f}) (\alpha) [\gamma_1;[(s(\gamma_1),\gamma');(s(\gamma'),\gamma'')]]&=& \alpha(r(\gamma_1),{g}({f}(\gamma_1)\gamma')\gamma'')\nonumber\\
&=& \alpha(r(\gamma'_1),({g}\circ{f})(\gamma_1) {g}(\gamma')\gamma'')
\end{eqnarray}
}
while
{
\begin{eqnarray}\label{rhs7}
((id\otimes \Xi_{g,f}^{X'})\circ \ep_{g\circ f})(\alpha)[\gamma_1;[(s(\gamma_1),\gamma');(s(\gamma'),\gamma'')]]&=& (\ep_{g\circ f}\alpha)[\gamma_1;(s(\gamma_1),{g}(\gamma')\gamma'')]\nonumber\\
&=& \alpha(r(\gamma_1),({g}\circ {f})(\gamma_1){g}(\gamma')\gamma'')
\end{eqnarray}
}
Thus from Equations $\ref{lhs7}$ and $\ref{rhs7}$ we get the desired equality.
\end{proof}

From the previous two lemmas, we can deduce:

\begin{proposition}\label{SmoothCommutative}
The following diagram is commutative
\end{proposition}
{
\medskip
\begin{picture}(415,80)
\put(90,60){$C_c^\infty ({\mathcal{G}}_{X''},r^\ast E'')$}
\put(135,50){ $\vector(0,-1){25}$}
\put(30,10){$C_c^\infty (\cG_{X}, r^*E)\otimes C_c^\infty (\cG_{X''}^{X}(g\circ f))$}
\put(85,38){$\ep_{g\circ f}\Psi_{g\circ f}$}
\put(200,70){$\ep_g\Psi_g$}
\put(180,64){\vector(1,0){50}}
\put(180,13){\vector(1,0){50}}
\put(185,25){$id \otimes \Xi_{g,f}^{X'}$}
\put(240,60){$C_c^\infty (\mathcal{G}_{X'}, r^*E') \otimes C_c^\infty (\mathcal{G}^{X'}_{X''}(g))$}
\put(280,50){ $\vector(0,-1){25}$}
\put(240,10){$C_c^\infty ({\cG}_{X}, r^*E) \otimes C_c^\infty (\cG_{X'}^{X} (f)) \otimes C_c^\infty (\cG_{X''}^{X''}(g)).$}
\put(300,38){$\ep_f\Psi_f \otimes id $}
\end{picture}
}
\vspace{-2cm}
\begin{Equation}
\end{Equation}
\vspace{1.5cm}\noindent

{
\begin{proof}
We know  from Lemma \ref{comdiag1} that
$$
\Pi_{g,f} \circ \Psi_{g\circ f}  = (\Psi_f\otimes id)\circ \ep_g\circ \Psi_g.
$$
Therefore,
$$
(\ep_f\Psi_f\otimes id) \circ \ep_g\Psi_g = (\ep_f\otimes id) \circ \Pi_{g,f} \circ \Psi_{g\circ f}.
$$
Now, from Lemma \ref{comdiag2}, we deduce that
$$
(\ep_f\otimes id) \circ \Pi_{g,f} = (id \otimes \Xi_{g,f}^{X'}) \circ \ep_{g\circ f}.
$$
Hence, we finally get
$$
(\ep_f\Psi_f\otimes id) \circ \ep_g\Psi_g = (id \otimes \Xi_{g,f}^{X'}) \circ \ep_{g\circ f}\circ \Psi_{g\circ f}.
$$
\end{proof}
}
We finish this subsection by deducing the proof of Theorem \ref{Functoriality}.
{
The composite map $(f_\phi^*\otimes id) \circ g_\phi^*$ is given by
$$
(f_\phi^*\otimes id) \circ g_\phi^* = (\ep_f\circ \Psi_f \circ \phi(\Delta) \otimes id) \circ \ep_g\circ \Psi_g\circ \phi(\Delta'').
$$
But the adjointable chain map $\ep_f\circ \Psi_f \circ \phi(\Delta') \otimes id$ induces the same map on cohomology as the adjointable chain map $ \phi(\Delta) \circ\ep_f\circ \Psi_f \otimes id$. Therefore, $(f_\phi^*\otimes id) \circ g_\phi^*$ induces the same map on cohomology as
$$
( \phi(\Delta)\otimes id\otimes id) \circ(\ep_f\circ \Psi_f \otimes id) \circ \ep_g\circ \Psi_g\circ \phi(\Delta'').
$$
Now, by Proposition \ref{SmoothCommutative}, we have on smooth compactly supported forms
$$
( \phi(\Delta)\otimes id\otimes id) \circ(\ep_f\circ \Psi_f \otimes id) \circ \ep_g\circ \Psi_g\circ \phi(\Delta'') = (\phi(\Delta)\otimes id\otimes id) \circ (id \otimes \Xi_{g,f}^{X'}) \circ (g\circ f)_\phi^*.
$$
Since this is an adjointable operator, this relation still holds on the Hilbert module $\maE_{X'', E''}$. Moreover, the operator $(id \otimes \Xi_{g,f}^{X'}) \circ (g\circ f)_\phi^*$ is an adjointable chain map and  $\phi(\Delta)\otimes id\otimes id$ induces the identity on cohomologies, whence $(f_\phi^*\otimes id) \circ g_\phi^*$ induces the same map on cohomologies as the map $(id \otimes \Xi_{g,f}^{X'}) \circ (g\circ f)_\phi^*$.
}

\section{leafwise homotopy equivalence and HP complexes}\label{HomotopyModules}
This section is devoted to the main result of this paper, namely that any leafwise homotopy equivalence induces a homotopy of the corresponding HP complexes, and hence an explicit homotopy between the corresponding leafwise signature operators.

\subsection{Leafwise homotopy equivalences}

\begin{proposition}\label{Iso}\
Assume that $(V, \maF)$  and $(V', \maF')$ are closed oriented foliated manifolds and that $f$ is an oriented leafwise homotopy equivalence between $(V,\maF)$ and $(V',\maF')$ with homotopy inverse $g$. We fix as before complete transversals $X$ and $X'$  in $(V,\maF)$ and $(V', \maF')$ respectively,   and we also consider another complete transversal $X''$ in  $(V,\maF)$. Then we have an isomorphism of Hilbert modules
$$
\mathcal{E}_{X'}^X(f)\otimes_{C^\ast(\mathcal{G}'^{X'}_{X'})}\mathcal{E}^{X'}_{X''}(g)\cong  \mathcal{E}^{X''}_{X}(g\circ f).
$$
\end{proposition}

\begin{proof}
We use the previous notations  in the proof of Proposition \ref{composition} and compute for $\xi_f\in C_c(\cG_{X'}^X(f))$ and $\eta_g\in C_c(\cG_{X''}^{X'} (g)$, $<\xi_f \ast \eta_g, \xi_f \ast \eta_g>$ as follows. Let us thus denote by $L$, $L'$ and $L''$ the leaves in $V$, $V'$ and $V''$ corresponding to a fixed $\gamma''\in {\cG''}_{X''}^{X''}$, see comment after the proof. Then we can write
$$
<\xi_f\ast \eta_g, \xi_f\ast \eta_g>(\gamma'')=\sum_{x\in L\cap X} \sum_{\gamma''_1 \in {\cG''}_{r(\gamma'')}^{gf(x)}} \overline{\xi_f\ast\eta_g(x,\gamma''_1)}\xi_f\ast\eta_g(x, \gamma''_1\gamma'')
$$
Replacing $\xi_f \ast \eta_g$ by its definition, we get
\begin{multline*}
 <\xi_f\ast \eta_g, \xi_f\ast \eta_g>(\gamma'')=\sum_{x\in L\cap X} \sum_{\gamma''_1 \in {\cG''}_{r(\gamma'')}^{gf(x)}}
\sum_{\gamma'_1 \in {\cG'}_{X'}^{f(x)}} \overline{\xi_f (x, \gamma'_1) \eta_g(s(\gamma'_1), g({\gamma'}_1^{-1}\gamma''_1)}\\
\sum_{\gamma'_2\in {\cG'}_{X'}^{f(x)}} \xi_f(x, \gamma'_2) \eta_g(s(\gamma'_2), g({\gamma'}_2^{-1})\gamma''_1\gamma'').
\end{multline*}
We hence get
\begin{multline*}
  <\xi_f\ast \eta_g, \xi_f\ast \eta_g>(\gamma'')= \sum_{x\in L\cap X} \sum_{x', y'\in L'\cap X'} \sum_{\gamma''_1\in {\cG''}_{r(\gamma'')}^{g(x')}} \sum_{\gamma'_2\in {\cG'}_{x'}^{f(x)} } \sum_{\alpha'\in {\cG'}_{y'}^{x'}}\\
\overline{\xi_f (x, \gamma'_2) \eta_g(x', \gamma''_1)} \eta_g(y', g(\alpha')^{-1}\gamma''_1\gamma'') \xi_f(x, \gamma'_2\alpha').
\end{multline*}
In the  sum over $\alpha'$ we set $\gamma'_3= \gamma'_2\alpha'$ and  thus get
\begin{multline*}
  <\xi_f\ast \eta_g, \xi_f\ast \eta_g>(\gamma'')= \sum_{x\in L\cap X} \sum_{x', y'\in L'\cap X'} \sum_{\gamma''_1\in {\cG''}_{r(\gamma'')}^{g(x')}} \sum_{\gamma'_2\in {\cG'}_{x'}^{f(x)} } \sum_{\gamma'_3\in {\cG'}_{y'}^{f(x)}}\\
\overline{\xi_f (x, \gamma'_2) \eta_g(x', \gamma''_1)} \eta_g(y', g(\gamma'_3)^{-1}g(\gamma'_2) \gamma''_1\gamma'') \xi_f(x, \gamma'_3).
\end{multline*}
Now computing similarly $<\eta_g, \pi_g(<\xi_f, \xi_f>)(\eta_g)> (\gamma'')$, we get exactly the same expression.

We now need to check surjectivity. But Proposition \ref{Compacts} shows that the representation
$$
\pi_{g\circ f}:C^\ast(\mathcal{G}^{X}_{X})\longrightarrow \mathcal{K}_{C^*(\maG_{X''}^{X''})}(\mathcal{E}_{X}^{X''}(g\circ f)),
$$
 is a $C^*$-algebra isomorphism. On the other hand, for any $\xi\in C_c(\cG_{X''}^X(g\circ f))$, we have
$$
\pi_{g\circ f}(\eta_1\star \eta_2) \xi = \eta_1 \ast (\eta_2 \bullet \xi),
$$
where
$$
\eta_2\bullet \xi (x', \gamma''):= \sum_{x\in L\cap X} \sum_{\gamma'\in {\cG'}_{x'}^{f(x)}} {\overline{\eta_2(x, \gamma')}} \xi(x, g(\gamma')\gamma'').
$$
We compute the left hand side as follows:
\begin{eqnarray}\label{lhs10}
\pi_{g\circ f}(\eta_1\star \eta_2)\xi  &=&  \sum_{\alpha\in \mathcal{G}^x_X} (\eta_1\star\eta_2)(\alpha) \xi(s(\alpha), {g}\circ {f}(\alpha^{-1}) \gamma'')\nonumber\\
&=& \sum_{\alpha\in \mathcal{G}^x_X} \sum_{\alpha'_1\in \mathcal{G}'^{f(x)}_{X'}}\eta_1(x,\alpha'_1)\overline{\eta_2(s(\alpha),{f}(\alpha^{-1})\alpha'_1)}  \xi(s(\alpha), {g}\circ {f}(\alpha^{-1}) \gamma'')
\end{eqnarray}
Now computing the right hand side, we have for any $\gamma'\in \mathcal{G}'^{f(x)}_{X'}$,
\begin{eqnarray}
\label{rhs10}
[\eta_1\ast (\eta_2\bullet \xi)](x,\gamma'')&=& [\eta_1\ast (\eta_2\bullet \xi)][(x,\gamma');(s(\gamma'),{g}(\gamma'^{-1})\gamma'')] \nonumber\\
&=& \sum_{\alpha'\in \mathcal{G}'^{s(\gamma')}_{X'}} \eta_1(x,\gamma'\alpha') (\eta_2\ast \xi)(s(\alpha'),{g}(\alpha'^{-1}\gamma'^{-1}) \gamma'')\nonumber\\
&=& \sum_{\alpha'_1\in \mathcal{G}'^{f(x)}_{X'}} \eta_1(x,\alpha'_1) (\eta_2\ast \xi)(s(\alpha'),{g}(\alpha'^{-1}_1) \gamma'')\nonumber\\
&=& \sum_{\alpha'_1\in \mathcal{G}'^{f(x)}_{X'}} \eta_1(x,\alpha'_1) \sum_{x_1\in X\cap L_{s(\alpha'_1)}} \sum_{\gamma'_1\in \mathcal{G'}^{f(x_1)}_{s(\alpha'_1)}} \overline{\eta_2(x_1,\gamma'_1)} \xi(x_1,{g}(\gamma'_1\alpha'^{-1}_1)\gamma'') \nonumber\\
&=& \sum_{\alpha'_1\in \mathcal{G}'^{f(x)}_{X'}} \eta_1(x,\alpha'_1) \sum_{x_1\in X\cap L_{s(\alpha'_1)}} \sum_{\gamma'_2\in \mathcal{G'}^{f(x_1)}_{f(x)}} \overline{\eta_2(x_1,\gamma'_2\alpha'_1)} \xi(x_1,{g}(\gamma'_2)\gamma'') \nonumber\\
&=& \sum_{\alpha'_1\in \mathcal{G}'^{f(x)}_{X'}} \eta_1(x,\alpha'_1) \sum_{x_1\in X\cap L_{x}} \sum_{\gamma_2\in \mathcal{G}^{x_1}_{x}} \overline{\eta_2(x_1,{f}(\gamma_2)\alpha'_1)} \xi(x_1,{g}\circ {f}(\gamma_2))\gamma'') \nonumber\\
&=& \sum_{\alpha'_1\in \mathcal{G}'^{f(x)}_{X'}} \eta_1(x,\alpha'_1)  \sum_{\gamma_2\in \mathcal{G}^{X}_{x}} \overline{\eta_2(r(\gamma_2),{f}(\gamma_2)\alpha'_1)} \xi(r(\gamma_2),{g}\circ {f}(\gamma_2))\gamma'') \nonumber\\
&=& \sum_{\alpha'_1\in \mathcal{G}'^{f(x)}_{X'}} \eta_1(x,\alpha'_1)  \sum_{\gamma\in \mathcal{G}^{x}_{X}} \overline{\eta_2(s(\gamma),{f}(\gamma^{-1})\alpha'_1)} \xi(s(\gamma),{g}\circ {f}(\gamma^{-1}))\gamma'') \nonumber\\
&=& \sum_{\alpha'_1\in \mathcal{G}'^{f(x)}_{X'}} \sum_{\gamma\in \mathcal{G}^{x}_{X}} \eta_1(x,\alpha'_1)  \overline{\eta_2(s(\gamma),{f}(\gamma^{-1})\alpha'_1)} \xi(s(\gamma),{g}\circ {f}(\gamma^{-1}))\gamma'')
\end{eqnarray}
Comparing $\eqref{lhs10}$ and $\eqref{rhs10}$ gives the desired equality.
\end{proof}

We have used in the previous proof the following standard results (see also \cite{RoyThesis}):
\begin{itemize}
 \item Let $f$ be an oriented homotopy equivalence between the oriented closed manifolds $V$ and $V'$, then $f$ is surjective.
\item Assume that $f$ is a leafwise homotopy equivalence between $(V,\maF)$ and $(V',\maF')$ then for any $(x,y)\in V^2$ lying in the same leaf, the induced map between the monodromy groupoids $\cG$ and $\cG'$ restricts  to a bijection between $\cG_x^y$ and  ${\cG'}_{f(x)}^{f(y)}$.
\item Assume that $f$ is a leafwise homotopy equivalence between $(V,\maF)$ and $(V',\maF')$ then the inverse image of a given leaf is a single leaf.
\end{itemize}

\begin{corollary}\label{Morita}
Under the assumption of Proposition \ref{Iso} we have an isomorphism of Hilbert $C^\ast(\mathcal{G}^X_X)$-modules
$$
\mathcal{E}_{X'}^X(f)\otimes_{C^\ast(\mathcal{G}'^{X'}_{X'})}\mathcal{E}^{X'}_{X}(g)\cong  C^*(\cG ^{X}_{X}).
$$
In the same way we have an isomorphism of Hilbert $C^\ast(\mathcal{G'}^{X'}_{X'})$-modules
$$
\mathcal{E}_{X}^{X'}(g)\otimes_{C^\ast(\mathcal{G}^{X}_{X})}\mathcal{E}^{X}_{X'}(f)\cong  C^*({\cG'}^{X'}_{X'}).
$$
\end{corollary}

\begin{proof}
We only prove the first isomorphism. Recall that $g\circ f$ is then leaf-preserving and that there exists a smooth leaf-preserving homotopy $H:V\times [0,1]\to V$ such that
$$
H(0, \bullet) = id \text{ and } H(1, \bullet) = g\circ f.
$$
Here the foliation on $V\times [0,1]$ is the one with leaves $L\times [0,1]$ where $L$ is a leaf of $(V, \maF)$. For any $x\in V$ the homotopy class of the path $(H(t,x))_{0\leq t\leq 1}$ is denoted $\gamma_x$ so $\gamma_x\in \cG_x^{gf(x)}$. We then consider the following map
$$
\Lambda: C_c(\cG_X^X(g\circ f) ) \rightarrow C_c(\cG_X^X)\text{ given by } \Lambda\xi(\gamma) := \xi( r(\gamma), \gamma_{r(\gamma)}\gamma ).
$$
A straightforward computation shows that $\Lambda$ allows to identify the Hilbert $C^*(\cG_X^X)$-modules $\maE_X^X(g\circ f)$ and $C^*(\cG_X^X)$. Applying
Proposition \ref{Iso}, we conclude.
\end{proof}

For $0 \leq s\leq 1$, let $h_s:= h\circ i_s$, where $i_s: V\hookrightarrow [0,1]\times V$ is the map $i_s(v)= (s,v)$. Then  the arguments in the proof of Corollary \ref{Morita} above give

\begin{proposition}\label{homotopyiso}
$\mathcal{E}^X_X(h_s)$ is isomorphic to $C^*(\cG_X^X)$ as a Hilbert module for all $0\leq s\leq 1$.
\end{proposition}

\begin{proof}\
Since the techniques are the same, we shall be brief. Denote more generally for $0\leq s\leq 1$ and for $x\in V$  by $\gamma^s_x$ the homotopy class of the path $t\rightarrow h(t,x), 0\leq t \leq s$.
We define a map $\theta^s_h: C_c(\mathcal{G}^X_X)\rightarrow C_c(\mathcal{G}^X_X(h_s))$ by the following formula:
$$
\theta^s_h(\xi)(x,\gamma)=\xi((\gamma^s_x)^{-1}\gamma), \quad \text{ for }\xi \in C_c(\mathcal{G}^X_X),
$$
Then a direct inspection shows that $\theta^s_h$ is $C_c(\mathcal{G}^X_X)$-linear. Moreover, $\theta^s_h$ is easily proved to be an isometry. To finish the proof, we set for $\eta\in C_c(\mathcal{G}^X_X(g\circ f))$ and for $\gamma\in \mathcal{G}_X^X$,
$$
(\theta^s_h)^{\ast}\eta(\gamma)=\eta(r(\gamma), \gamma^s_{r(\gamma)}\gamma).
$$
Then $(\theta^s_h)^{\ast}\eta\in C_c(\mathcal{G}^X_X)$, and we have,
$$
\theta^s_h((\theta^s_h)^{\ast}\eta)(x,\gamma)=((\theta^s_h)^{\ast}\eta)((\gamma^s_x)^{-1}\gamma)= \eta(x,\gamma^s_x(\gamma^s_x)^{-1}\gamma)=\eta(x,\gamma).
$$
Thus $\theta^s_h$ is surjective.
\end{proof}

\begin{remark}\
The map  $\mathcal{G}^X_X(id_V)\stackrel{\pi_2}{\rightarrow} \mathcal{G}^X_X$ induced by  the projection onto the second factor is a diffeomorphism.
\end{remark}

We point out that for any Hilbert $C^*(\cG_X^X)$-module $\maE$, the Hilbert module $\mathcal{E}\otimes_{C^*(\cG_X^X)} C^*(\cG_X^X)$ is canonically isomorphic  to $\mathcal{E}$.  We denote this canonical isomorphism generically by $\delta$.

Assume again that $f$ is an oriented leafwise homotopy equivalence with homotopy inverse $g$ and let us briefly describe  the relation between our Morita Hilbert modules and those introduced by Connes and Skandalis. In \cite{Connes81, ConnesSkandalis}, the graph $\mathcal{G}(f)$ and a corresponding Hilbert module are defined. More precisely, the Connes-Skandalis graph is given by:
$$
\mathcal{G}(f):=\{(v,\alpha'); v\in V, \alpha'\in \mathcal{G}'\text{ and } f(x)=r(\alpha')\}.
$$
$\mathcal{G}(f)$ is a right principal $\mathcal{G}'$-bundle, and it also has an action of $\mathcal{G}$ on the left. One then defines a Hilbert module $\maE(f)$ that we shall call the Connes-Skandalis module.
Let $\maE_X$ be the Hilbert $C^*(\cG_X^X)$-module which is the completion of $C_c(\cG_X)$. We define in the same way the Hilbert $C^*({\cG'}_{X'}^{X'})$-module $\maE_{X'}$. Then the expected relation between the four Hilbert modules can be proved, see \cite{RoyThesis}, i.e.

\begin{proposition}
$$\label{compatibility}
\mathcal{E}_{X}{\otimes}_{C^\ast(\mathcal{G}^X_X)}\mathcal{E}^X_{X'}(f)\cong \mathcal{E}(f){\otimes}_{C^\ast(\mathcal{G}')} \mathcal{E}_{X'}.
$$
\end{proposition}

\begin{proof}
We have from Proposition \ref{Iso1}, $ \mathcal{E}_{X}{\otimes}_{C^\ast(\mathcal{G}^X_X)}\mathcal{E}^X_{X'}(f) \cong \maE_{X'}(f)$, so it suffices
to show that $\mathcal{E}(f){\otimes}_{C^\ast(\mathcal{G}')} \mathcal{E}_{X'}\cong \maE_{X'}(f)$. We define a map
$\upsilon: C^\infty_c(\mathcal{G}(f)) \otimes_{C_c^\infty(\mathcal{G}_X^X)} C_c^\infty(\mathcal{G}_{X'})\rightarrow C^\infty_c(\mathcal{G}_{X'}(f))$ as follows:
 $$
\upsilon(\xi\otimes\eta)(v,\gamma'):= \int_{\alpha'\in \mathcal{G}'^{r(\gamma')}} \xi(v, \alpha')\eta(\alpha'^{-1}\gamma') d\lambda^{r(\gamma')} (\alpha')
\text{ for $\xi \in C^\infty_c(\mathcal{G}(f)), \eta \in C_c^\infty(\mathcal{G}_{X'})$} $$

One can check that the above map is well-defined and the following property holds (see \cite{RoyThesis}):
$$
<\upsilon(\xi\otimes \eta), \upsilon(\xi\otimes \eta)>= <\eta,\chi_m(<\xi,\xi>) \eta>,
$$
where $\chi_m: C^*(\maG')\rightarrow \maK(\maE_{X'})$ is an isomorphism (cf. \cite{HilsumSkandalis}). The above property then implies that $\upsilon$
induces an isometry on the level of Hilbert-modules $\upsilon: \mathcal{E}(f) \otimes_{C^*(\mathcal{G}_X^X)} \mathcal{E}_{X'}\rightarrow \mathcal{E}_{X'}(f)$.

Lastly, in order to prove surjectivity of $\upsilon$, we proceed as follows (cf. \cite{ConnesSkandalis}, Proposition 4.5). Since $\maE(f)$ implements the
Morita equivalence between $C^*(\maG)$ and $C^*(\maG')$, we have an isomorphism $\pi_f: C^*(\maG)\rightarrow \maK(\maE(f))$ ( \cite{HilsumSkandalis} ,Corollary 7)
given by the following formula:
$$
\pi_f(\phi)\xi(v,\gamma')=  \int_{\alpha\in \maG^v} \phi(\alpha)\psi(s(\alpha),f(\alpha^{-1})\gamma') d\lambda^v(\alpha) \text{ for } \phi \in C_c^\infty(\maG),
\psi \in C_c^\infty(\maG(f)).
$$
Let $\xi_1,\xi_2\in C^\infty_c(\maG(f))$. Let $\xi_1\star \xi_2$ denote the function on $\maG$ given by
$$
 \xi_1\star \xi_2(\gamma)= \int_{\alpha' \in \maG'^{f(r(\gamma))}} \xi_1(r(\gamma), \alpha') \overline{\xi_2(s(\gamma), f(\gamma^{-1})\alpha')}
d\lambda^{f(r(\gamma))}(\alpha')
$$
Denote by $\theta_{\xi_1,\xi_2}$ the operator in $\maK(\maE(f))$ given by $\theta_{\xi_1,\xi_2}\zeta:= \xi_1<\xi_2,\zeta>$. Then a direct calculation shows that
$\theta_{\xi_1,\xi_2}= \pi_f(\xi_1\star \xi_2)$. Now, we also have a representation $\pi(f): C^*(\maG)\rightarrow \maK(\maE^V_{X'}(f))$ defined by
$$
[\pi(f)(\phi)\psi](v,\gamma') = \int_{\alpha\in \maG^v} \phi(\alpha)\psi(s(\alpha),f(\alpha^{-1})\gamma') d\lambda^v(\alpha) \text{ for }
\phi \in C_c^\infty(\maG), \psi \in C_c^\infty(\maG_{X'}(f)).
$$
Now, as in Proposition \ref{Compacts}, we can easily prove the following equality:
$$
\theta_{\phi_1,\phi_2}\psi= \pi(f)(\phi_1\star \phi_2)\psi \text{ for } \phi_1,\phi_2,\psi \in C_c^\infty(\maG_{X'}(f)),
$$
where
$$
\phi_1\star \phi_2(\alpha)= \sum_{\alpha'_1 \in \maG'^{f(r(\alpha))}_{X'}} \phi_1(r(\alpha),\alpha'_1)\overline{\phi_2(s(\alpha),f(\alpha^{-1})\alpha_1)}
\text{ for } \alpha\in \maG.
$$
This implies that $\pi(f)(C^*(\maG))\maE_{X'}(f)$ is dense in $\maE_{X'}(f)$. Therefore, to prove surjectivity it suffices to show that an element of the form $\pi(f)(h)\xi$
lies in the range of $\upsilon$ for all $h\in C^*(\maG)$ and $\xi \in \maE_{X'}(f)$. This is done by a straightforward calculation to prove that
$$
\pi(f)(\xi_1\star\xi_2)\kappa= \upsilon(\xi_1\otimes (\xi_2\bullet \kappa)) \text{ for } \xi_1,\xi_2\in C^\infty_c(\maG(f)), \kappa \in C^\infty_c(\maG_{X'}(f)),
 $$
where
$$
\xi_2\bullet\kappa(\gamma')= \int_{v\in L_{\gamma'}} \sum_{\gamma'_1\in \maG'^{f(v)}_{r(\gamma')}} \overline{\xi_2(v,\gamma'_1)} \kappa(v,\gamma'_1\gamma')
d\lambda^{L_{\gamma'}}(v) \text{ for } \gamma'\in \maG'^{X'}_{X'}.
$$
We refer the reader to \cite{RoyThesis} for the detailed computations.
\end{proof}

\subsection{Poincar\'{e} lemma for foliations}\label{PoincareLemma}

Again $f:(V,\maF)\to (V', \maF')$ is a smooth oriented leafwise homotopy equivalence with homotopy inverse $g$. We denote by $h: [0,1]\times V\to V$ and $h':[0,1]\times V'\to V'$ the oriented smooth leaf-preserving homotopies between $id_V$ and $g\circ f$ on the one hand, and between $id_{V'}$ and $f\circ g$ on the other hand. So,
$$
h(0, \bullet)= id_V, h(1, \bullet)= g\circ f, h'(0, \bullet) = id_{V'} \text{ and } h'(1, \bullet)= f\circ g.
$$
The goal of the present subsection is to prove the following

\begin{theorem}\label{isomorphism}
The maps $f_\phi^*$ and $g_\phi^*$ induce isomorphisms in cohomology, which are inverse of each other when we identify cohomologies using the Morita isomorphisms described in \ref{Morita}. Said differently, an oriented leafwise homotopy equivalence induces a homotopy of the corresponding HP complexes.
\end{theorem}

We set as before, $h_s(v):= h(s, v)$ and $h'_s(v'):= h'(s, v')$ for any $s\in [0,1]$. Recall also that we have defined isometric isomorphisms of Hilbert $C^*(\cG_X^X)$-modules (resp. Hilbert $C^*({\cG'}_{X'}^{X'})$-modules):
$$
\theta_h^s: C^*(\cG_X^X) \rightarrow \maE_X^X (h_s) \;  (\text{resp. } \theta_{h'}^s: C^*({\cG'}_{X'}^{X'}) \rightarrow \maE_{X'}^{X'} (h'_s)).
$$
The notations $\theta_h$ and $\theta_{h'}$ stand for $\theta_h^1$ and $\theta_{h'}^1$ respectively.

\begin{lemma}\label{homiso}
The map $\theta_h^s$ is equivariant with respect to the left representation of the $C^*$-algebra $C^*(\cG_X^X)$ and
the well defined composite map
$$
\rho^s_h := \delta\circ (I\otimes_{C^*(\cG_X^X)} (\theta_h^s)^{-1}): \mathcal{E}_{X,E}\otimes_{C^*(\cG_X^X)} \mathcal{E}^X_X(h_s)\longrightarrow \mathcal{E}_{X,E}
$$
commutes with the differential on  $\mathcal{E}_{X,E}$.
\end{lemma}

\begin{proof}
Let $\phi, \xi\in C_c(\cG_X^X)$. Then,
\begin{eqnarray*}
\theta^s_h(\phi\ast \xi)(x,\gamma) &=& (\phi\ast\xi)((\gamma_x^s)^{-1}\gamma)\nonumber\\
&=& \sum_{\gamma_1\in \mathcal{G}^x_X} \phi(\gamma_1)\xi(\gamma_1^{-1}(\gamma^s_x)^{-1}\gamma)
\end{eqnarray*}
On the other hand,
$$
\pi(\phi)\theta^s_h(\xi)(x,\gamma) = \sum_{\alpha\in \mathcal{G}^x_X} \phi(\alpha)\xi((\gamma^s_{s(\alpha)})^{-1}\circ ({f}\circ{g})(\alpha^{-1})\circ \gamma).
$$
Putting $\gamma_1= (\gamma^s_x)^{-1} ({f}\circ {g})(\alpha) \gamma^s_{s(\alpha)}$ we see from the proof of {Proposition 5.1.7} of
\cite{RoyThesis}
 that $\gamma_1=\alpha$ and hence,
$$
\pi(\phi)\theta^s_h(\xi)(x,\gamma) = \sum_{\gamma_1\in \mathcal{G}^x_X} \phi(\gamma_1)\xi(\gamma_1^{-1} (\gamma^s_x)^{-1} \gamma).
$$
To see that $\rho^s_h$ is a chain map, we compute as follows:
\begin{eqnarray*}
\rho^s_h\circ (\td \otimes I)&=& \delta\circ (I\otimes (\theta^s_h)^{-1})\circ (\td\otimes I)\\
&=& \delta\circ  (\td\otimes I)\circ (I\otimes (\theta^s_h)^{-1})
\end{eqnarray*}
But since $\td$ is $C_c^\infty(\cG_X^X)$-linear, we have:
$$
\delta\circ (\td\otimes I) = \td \circ \delta.
$$
Therefore from the two computations above we get the desired result.
\end{proof}

Let $X_0:=\{0\}\times X \subset [0,1]\times V$ be the complete transversal of the product foliation on $[0,1]\times V$. We denote for simplicity this foliation by $[0,1]\times \mathcal{F}$. Let $\hat{\mathcal{G}}$ be the monodromy groupoid of the foliation $([0,1]\times V,[0,1]\times \mathcal{F})$. Then  $\hat{\mathcal{G}}$ can be, and will be, identified as a smooth groupoid with the cartesian product $\mathcal{G}\times [0,1]^2$ of the groupoid $\cG$ with the product groupoid $[0,1]^2$.

Define maps $\ep_h$ and $\Psi_h$ for the leafwise map $h$, as we did in Section \ref{HomotopyModules} for $f$. We note that using a similar proof as for Lemma \ref{homotopyiso}, one has an isomorphism
$$
\rho_h: \mathcal{E}_{X_0,\hat{E}}\otimes_{C^*(\hat\cG_{X_0}^{X_0})}\mathcal{E}^{X_0}_X(h)\longrightarrow \mathcal{E}_{X_0,\hat{E}} \text{ where }\hat{E}:= \Lambda^* T^*([0,1]\times \mathcal{F}).
$$
We set $H^*:= \rho_h \circ \ep_h \circ \Psi_h$,
 so this is the map from $\maE_{X,E}$ to $\maE_{X_0, {\hat E}}$ given by
$$
\mathcal{E}_{X,E}\stackrel{\Psi_h}{\rightarrow} \mathcal{E}^{[0,1]\times V}_{X, {\hat E}}(h)\stackrel{\ep_h}{\rightarrow} \mathcal{E}_{X_0,\hat{E}}\otimes_{C^*(\hat\cG_{X_0}^{X_0})} \mathcal{E}^{X_0}_X(h)  \stackrel{\rho_h}{\rightarrow} \mathcal{E}_{X_0,\hat{E}}.
$$
Notice that $\hat{\mathcal{G}}_{X_0}$ is identified with $[0,1]\times \mathcal{G}_X$ while  $\hat{\mathcal{G}}^{X_0}_{X_0}$ is identified with $\mathcal{G}^X_X$. We finally get as in Section \ref{HomotopyModules} a well defined adjointable chain map
$$
H_\phi^\sharp:=H^*\circ \phi(\Delta) = \rho_h \circ H_\phi^* : \mathcal{E}_{X,E}\longrightarrow \mathcal{E}_{X_0,\hat{E}}.
$$
We define in the same way $(h_s)_\phi^\sharp: \mathcal{E}_{X,E}\longrightarrow \maE_{X,E}$ for any $s\in [0,1]$.
Notice that integration over $(0,1)$ yields a well defined degree $-1$ linear map
$$
\int_{(0,1)} : C_c^\infty (\hat\cG_{X_0}, r^*\hat{E}) \longrightarrow C_c^\infty (\cG_X, r^*E).
$$
It is clear from its very definition  that $\int_{(0,1)}$ is $C_c^\infty (\cG_X^X)$-linear when we identify as we did $\hat\cG_{X_0}^{X_0}$ with $\cG_X^X$. Moreover, a direct inspection shows that $\int_{(0,1)}$ extends to an adjointable operator
$$
\int_{0,1)} : \maE_{X_0, \hat E} \longrightarrow \maE_{X,E},
$$
with the adjoint given by the extension of $\eta\mapsto \eta \wedge dt$ where $t$ is the variable in $[0,1]$ and $\eta$ is understood via its usual pull-back. We set
$$
K_\phi^\sharp := \int_{(0,1)} \circ H_\phi^\sharp : \mathcal{E}_{X,E}\longrightarrow \maE_{X,E}.
$$
So, $K_\phi^\sharp$ is an adjointable operator.

\begin{lemma}
The following relation holds
$$
(g\circ f)_\phi^\sharp - \phi(\Delta) = K_\phi^\sharp \circ d_X + d_X \circ K_\phi^\sharp.
$$
\end{lemma}

\begin{proof}
First notice that for any $x\in X$, the map
$$
H^*_x:C_c^\infty(\mathcal{G}_x,r^*E)\longrightarrow C_c^\infty(\hat{\mathcal{G}}_{(0,x)},r^*\hat{E})=  C_c^\infty([0,1]\times \mathcal{G}_{x},r^*\hat{E}).
$$
coincides with the composite map
$$
C_c^\infty(\mathcal{G}_X,r^* E) \xrightarrow{\Psi_h} C_c^\infty(\mathcal{G}^{[0,1]\times V}_X(h), (r\circ pr_1)^* \hat{E}) \xrightarrow{\Phi_h} C_c^\infty(\hat{\mathcal{G}}_{X_0}\times_{\hat{\mathcal{G}}_{X_0}^{X_0}} \mathcal{G}^{X_0}_X(h), (r\circ pr_1)^* \hat{E} ) \xrightarrow{\hat\rho_h }  C_c^\infty(\hat{\mathcal{G}}_{X_0},\hat{E} )
$$
where we set $\hat\rho_h=\rho_h\circ \nu_h^{-1}$ and use the notations of the previous section.
We compute
\begin{eqnarray*}
(\Phi_h\circ \Psi_h) (\xi)[(t,\gamma), (s(t,\gamma),\alpha)]&=& \Psi_h(\xi)(r(t,\gamma), {h}(t,\gamma)\alpha)  \\
&=& (^t{h}_\ast)_{(t,r(\gamma))} (\xi({h}(t,\gamma)\alpha) )
\end{eqnarray*}
Therefore, since  $\theta_h : \mathcal{G}^{X_0}_X(h)\rightarrow \mathcal{G}^{X_0}_{X_0}$ is given by $((0,x),\gamma)\mapsto (0,0,\gamma)$, we get
\begin{eqnarray*}
(I\otimes \theta_h^{-1})\circ  (\ep_h\circ \Psi_h) (\xi)[(t,\gamma), (0,\alpha)]&=& (\Phi_h\circ \Psi_h) (\xi)[(t,\gamma), (s(t,\gamma), (0,\alpha))]\\
&=& (^t{h}_\ast)_{(t,r(\gamma)} (\xi({h}(t,\gamma)\alpha) )
\end{eqnarray*}
Hence,
\begin{eqnarray*}
H^* (\xi)(t,\gamma)&=& (I\otimes \theta_h^{-1})(\Phi_h\circ \Psi_h) (\xi)[(t,\gamma),1_{(0,\gamma)}] \\
&=& (^t{h}_\ast)_{(t,r(\gamma))} (\xi({h}(t,\gamma)1_{0,\gamma}) )\\
&=& (^t{h}_\ast)_{(t,r(\gamma))} (\xi(\theta(t,\gamma)\gamma^t_{s(\gamma)}) )
\end{eqnarray*}
We have used the following relations
$$
{h}(t,\gamma):= {h}(t,0,\gamma)= \gamma_{r(\gamma)}^t\gamma, \;\; h(t,s,\gamma)= \gamma_{r(\gamma)}^t\gamma(\gamma_{s(\gamma)}^s)^{-1},\;\; \theta(t,\gamma):= \gamma^t_{r(\gamma)}\gamma(\gamma^t_{s(\gamma)})^{-1}\text{ and }
\gamma_{r(\gamma)}^s\gamma= h_t(\gamma) \gamma^s_{s(\gamma)}.
$$
Therefore, $H_x^*$ is simply given by the formula:
\begin{equation}\label{homlemma}\
 H^*_x(\xi)(t,\gamma)=(^t{h}_\ast)_{(t,r(\gamma)}(\xi(\theta(t,\gamma)\circ \gamma^t_{s(\gamma)})),
\end{equation}
which explains the notation.
Now, computing in local coordinates, we check that the usual Poincar\'e equation holds, more precisely, in our setting, we have  for $\xi=\xi_1 + dt\wedge \xi_2$ with $\xi_1, \xi_2\in C_c^\infty (\hat\cG_x, r^*{\hat E})$,
$$
\td \int_{(0,1)} \xi  + \int_{(0,1)} d_{\hat\cG_{(0,x)}} \xi = \xi_1\circ i_1 - \xi_1 \circ i_0.
$$
Here $i_s: \cG_x \hookrightarrow \hat\cG_{(0,x)}$ is the map $i_s(\gamma):=(s, 0, \gamma)$. Applying this relation to $\xi=H^*_x(\eta)$  for $\eta\in C_c^\infty(\mathcal{G}_x,r^\ast E)$, we get
$$
\int_{(0,1)} d_{\hat\cG_{(0,x)}} H^\ast_x(\eta)+ \td (\int_{(0,1)} H^\ast_x(\eta))= H^\ast(\eta)_1|_{\{1\}\times \mathcal{G}_x}- H^\ast(\eta)_1|_{\{0\}\times \mathcal{G}_x}.
$$
Here we have used the notation $H^*\eta= (H^*\eta)_1 + dt\wedge (H^*\eta)_2$.

But we know from Equation \eqref{homlemma} that
$$
H^\ast(\eta)(1,\gamma)=(^t{h_\ast})_{(1,r(\gamma))} [\eta(\theta(1,\gamma)\circ \gamma_{s(\gamma)})] \text{ and } h^\ast_1(\eta)(\gamma)=(^t{h_{1,\ast}})_{r(\gamma)} [\eta(h_1(\gamma)\circ \gamma_{s(\gamma)})].
$$
To finish the proof, we thus only need to check that $H^\ast(\eta)_1|_{(1,\gamma)} = h_1^\ast(\eta)(\gamma)$, for then, we deduce:
$$
h_1^\ast(\eta)- h_0^\ast(\eta) =  \int_{(0,1)} d_{\hat{\mathcal{G}}_{(0,x)}} H^\ast_x(\eta)+ d_{\hat{\mathcal{G}}_{(0,x)}}(\int_{(0,1)} H^\ast_x(\eta)),
$$
and composing with $\phi(\Delta)$ on the right would end the proof.

Now to prove that $H^\ast(\eta)_1|_{(1,\gamma)} = h_1^\ast(\eta)(\gamma)$, we first notice  that if $\xi= \xi_1+\xi_2\wedge dt$ then $\xi_1= (^t{i_{1,\ast}})\xi$.
As $h_1= h\circ i_1$, we have $(^t{i_{1,\ast}})_x\circ (^t{h_{\ast}})_{(1,x)} = (^t{h_{1,\ast}})_x$, and so
\begin{eqnarray*}
H^\ast(\eta)_1 |_{(1,\gamma)}&= & (^t{i_{1,\ast}})_{r(\gamma)} (H^\ast(\eta) (1,\gamma)) \\
&= & (^t{i_{1,\ast}})_{r(\gamma)} \circ (^t{h_{\ast}})_{(1,r(\gamma))} (\eta(h(1,\gamma)\gamma_{s(\gamma)})\\
&=& (^t{h_{1,\ast}})_{r(\gamma)} (\eta(h(1,\gamma)\gamma_{s(\gamma)}) \\
&=& h_1^\ast (\eta)(\gamma)
\end{eqnarray*}
and the proof is thus complete.
\end{proof}

\begin{corollary}
$(g\circ f)_\phi^\sharp$ and $(f\circ g)_\phi^\sharp$ induce the identity maps on the cohomologies of the complexes $(\mathcal{E}_X,\td)$ and $(\mathcal{E}'_{X'},\td')$ respectively.
\end{corollary}

\begin{proof}
The first assertion  is immediate from the previous proposition and the fact that $\phi(\Delta)$ induces the identity on cohomology, while $K^\sharp\circ d+ d\circ K^\sharp$ is zero on cohomology. The second one is a consequence of the first one permuting the roles of $f$ and $g$.
\end{proof}

\subsection{Compatibility with Poincar\'{e} duality}

In this section we prove the compatibility of the pullback map $f_\phi^*$ with the Poincar\'{e} duality operators $T'_{X'}$ and $T_X\otimes Id$ of the HP-complexes $(\maE_{X',E'},d'_{X'})$ and $(\maE_{X,E}\otimes \maE^{X}_{X'}(f),d_{X}\otimes Id)$, respectively, in order to achieve the proof that the map $f^*_\phi$ is indeed a homotopy equivalence of HP-complexes, as per Definition \ref{homeqHPC}. Let $(f^*_\phi)^\sharp$ be the adjoint of $f^*_\phi$. We will show that
\begin{proposition}
The maps $f^*_\phi T'_{X'} (f^*_\phi)^\sharp$ and $T_X\otimes Id$ induce the same map on cohomology.
\end{proposition}
\begin{proof}
Recall that $f^*_\phi= \epsilon_f\Psi^\phi_f$, with $\Psi^\phi_f= \Psi_f\circ \phi(\Delta'): \maE_{X',E'}\rightarrow \maE^V_{X',E}(f)$. We note that the map $\Psi^\phi_f :C^\infty_c(\maG'_{X'},r^*E')\rightarrow C^\infty_c(\maG^V_{X'}(f), \pi_1^*E)$ can also be described as $\pi_{2,f}^*\circ \phi(\Delta')$, where $\pi_{2,f}^*$ is the leafwise pullback map associated with
$\pi_{2,f}= \left[ (\pi_{2,f})_{x'}: \maG_{x'}^V(f)\rightarrow \maG'_{x'}\right]$, where $(\pi_{2,f})_{x'}(v,\gamma')= \gamma'$. Similarly, we define $\pi_{2,g}$ and its induced map $\pi^*_{2,g}$ associated with the map $g$. We shall denote the composition $\pi_{2,f}^*\circ \phi(\Delta')$ as $\pi_{2,f,\phi}^*$.

Now $\pi_{2,f}$ is a leafwise homotopy equivalence for the induced foliations on $\maG_{X'}^V(f)$ and $\maG'_{X'}$ with a homotopy inverse given for instance by $\lambda_f: \gamma'\mapsto (g(r(\gamma')), (f\circ g)(\gamma'))$. As a consequence, we have
$$
\int_{\maG^V_{x'}(f)}\Psi_f(\omega')= \int_{\maG^V_{x'}(f)}(\pi^*_{2,f})_{x'}(\omega')= \int_{\maG'_{x'}} \omega',
$$
for any leafwise top dimensional closed differential form $\omega'$ on $\maG'_{x'}$. The map $\lambda_f$ induces a pullback on forms, which we denote by $\lambda_f^*$. Similarly, one can define a map $\lambda_{g}$ and its induced map $\lambda^*_{g}$ associated with the map $g$. We shall denote the composition $\phi(\Delta')\circ \lambda_{f}^* $ as $\lambda_{f,\phi}^*$ and $\phi(\Delta)\circ \lambda_{g}^* $ as $\lambda_{g,\phi}^*$.

Consider the following map $\Gamma: \maE_{X,E}\otimes \maE^X_{X'}(f)\rightarrow \maE_{X',E'}$ given by the composition
$$
\maE_{X,E}\otimes \maE^X_{X'}(f)\xrightarrow{T_X\otimes id} \maE_{X,E}\otimes \maE^X_{X'}(f)\xrightarrow{g^*_\phi\otimes Id} \maE_{X',E'}\otimes \maE^{X'}_X(g)\otimes \maE^X_{X'}(f)\xrightarrow{\Lambda'} \maE_{X',E'}\xrightarrow{T'_{X'}} \maE_{X',E'}
$$
where $\Lambda'=\rho^1_{h'}\circ (Id\times (\Xi^{X}_{f,g})^{-1})$ and $\rho^s_{h'}$ is the isometric isomorphism of Hilbert modules
defined analogously as $\rho^s_{h}$ for $s\in[0,1]$, given in Lemma \ref{homiso}.

We claim that
\begin{equation}\label{Claim}
\Gamma= (-1)^{k(p-k)} (f^*_{\phi})^\sharp \text{ on  degree $k$-cohomology.}
\end{equation}

Assuming the claim \ref{Claim}, it is easy to conclude the proof. Indeed, we then have on cohomology of degree $k$:
\begin{eqnarray*}
f^*_{\phi}T'_{X'}(f^*_{\phi})^\sharp&=&f^*_{\phi}T'_{X'}[(-1)^{k(p-k)}\Gamma]\\
&=&  (-1)^{k(p-k)} f^*_{\phi}T'_{X'} T'_{X'} \rho^1_{h'} (Id\times (\Xi^{X}_{f,g})^{-1}) (g^*_{\phi}\otimes Id) (T_X\otimes Id)
\end{eqnarray*}
Since on cohomology we have, from Theorem \ref{Functoriality},
$(g^*_\phi\otimes Id)f^*_\phi= (Id\otimes \Xi^{X'}_{f,g})\circ (f\circ g)^*_\phi$, so $g^*_\phi\otimes Id= (Id\otimes \Xi^{X}_{f,g})\circ (f\circ g)^*_\phi\circ (f^*_\phi)^{-1}$ on cohomology. Hence we have again on cohomology,
\begin{multline*}
f^*_{\phi}[(-1)^{k(p-k)}(T'_{X'})^2] \rho^1_{h'}(Id\times (\Xi^{X}_{f,g})^{-1}) (g^*_{\phi}\otimes Id) (T_X\otimes Id)
=\\
 f^*_{\phi}\circ \rho^1_{h'} \circ (Id\times (\Xi^{X}_{f,g})^{-1}) (Id\otimes \Xi^{X}_{f,g})\circ (f\circ g)^*_\phi\circ (f^*_\phi)^{-1}(T_X\otimes Id)
= \\
f^*_{\phi} \circ \rho^1_{h'} \circ (f\circ g)^*_\phi\circ (f^*_\phi)^{-1}(T_X\otimes Id)
\end{multline*}
But, we know from the Poincar\'{e} Lemma for foliations that $\rho^1_{h'} \circ (f\circ g)^*_\phi$ induces the identity on cohomology. Therefore from the last line above we get the desired equality $f^*_{\phi}T'_{X'}(f^*_{\phi})^\sharp= T_X\otimes Id$ on cohomology.

It thus remains to prove \ref{Claim}. We compute  for $\psi_1 \in \maE_{X,E}, \psi_2\in \maE^X_{X'}(f), \eta_1 \in \maE_{X'E'}, \eta_2 \in \maE^{X'}_X(g)$ and $\eta_3 \in \maE^X_{X'}(f)$,
\begin{eqnarray*}
<\psi_1\otimes \psi_2, f^*_\phi\Lambda'(\eta_1\otimes\eta_2\otimes\eta_3)>&=& <\psi_1\otimes \psi_2, \epsilon_f \Psi^\phi_f\Lambda'(\eta_1\otimes \eta_2\otimes\eta_3)>\\
&=& <\psi_1\otimes \psi_2, \epsilon_f \Psi^\phi_f \rho^1_{h'} (Id\otimes (\Xi^X_{f,g})^{-1})(\eta_1\otimes \eta_2\otimes\eta_3)>
\end{eqnarray*}
But we have $f^*_\phi \circ \rho^1_{h'} = \rho^1_{h'} \circ (f^*_\phi\otimes Id)$. Therefore we get on cohomology,
\begin{eqnarray*}
<\psi_1\otimes \psi_2, \epsilon_f \Psi^\phi_f \rho^1_{h'} (Id\otimes (\Xi^X_{f,g})^{-1})(\eta_1\otimes \eta_2\otimes\eta_3)> & = &  <\psi_1\otimes \psi_2, \rho^1_{h'}(\epsilon_f \Psi^\phi_f \otimes Id)(Id\otimes (\Xi^X_{f,g})^{-1})(\eta_1\otimes \eta_2\otimes\eta_3)>\\
&=& <\psi_1\otimes \psi_2, \rho^1_{h'}(Id\otimes (\Xi^X_{f,g})^{-1})(\epsilon_f \Psi^\phi_f \otimes Id)(\eta_1\otimes \eta_2\otimes\eta_3)>\\
&=& <\psi_1\otimes \psi_2, \rho^1_{h'}((g\circ f)^*_\phi (g^*_\phi)^{-1})(\eta_1\otimes \eta_2\otimes\eta_3)>\\
&=& <\psi_1\otimes \psi_2, ((\Psi^\phi_g)^{-1}\epsilon^{-1}_g\otimes Id)(\eta_1\otimes \eta_2\otimes\eta_3)>\\
&=& <\psi_1\otimes \psi_2, \lambda_{g,\phi}^*\epsilon^{-1}_g(\eta_1\otimes \eta_2)\otimes\eta_3>\\
&=& <\epsilon_g\circ (\lambda_{g,\phi}^*)^\sharp\psi_1\otimes \psi_2, \eta_1\otimes \eta_2\otimes\eta_3>
\end{eqnarray*}
where $(\lambda_{g,\phi}^*)^\sharp$ is the adjoint of the operator $\lambda_{g,\phi}^*$ defined above. To compute $(\lambda_{g,\phi}^*)^\sharp$, we recall that the inner product on $C^\infty_c(\maG^{V'}_{X}(g),\pi_1^* E')$ is given,   for $\xi_1,\xi_2 \in C^\infty_c(\maG^{V'}_{X}(g),\pi_1^* E')$, by:
$$
<\xi_1,\xi_2>(\gamma)= \int_{\maG^{V'}_{r(\gamma)}(g)} \xi_1\wedge\star'_g \xi_2,
$$
where $\star'_g$ is the Hodge $\star$-operator on $\maG^{V'}_X(g)$. Now let $\eta\in C^\infty_c(\maG^{V'}_{X}(g),\pi_1^* E')$ and $ \xi   \in C^\infty_c(\maG_X,r^*E)$ be closed $k$-forms.
Then we have,
\begin{eqnarray*}
<\lambda_{g,\phi}^*\eta,\xi>(\gamma)
&=& \int_{\maG_{r(\gamma)}} (\lambda^*_{g,\phi})_{r(\gamma)}\eta \wedge\star \xi \\
&=& \int_{\maG_{r(\gamma)}}  (\lambda^*_{g,\phi})_{r(\gamma)}\eta \wedge (\lambda^*_{g,\phi})_{r(\gamma)} \circ (\pi_{2,g,\phi}^*)_{r(\gamma)} \star \xi\\
&=& \int_{\maG_{r(\gamma)}} (\lambda^*_{g,\phi})_{r(\gamma)} ( \eta \wedge (\pi_{2,g,\phi}^*)_{r(\gamma)} \star \xi) \\
&=& \int_{\maG^{V'}_{r(\gamma)}(g)} \eta \wedge(\pi_{2,g,\phi}^*)_{r(\gamma)}\star \xi \\
&=& \int_{\maG^{V'}_{r(\gamma)}(g)} \eta \wedge \star'_g((-1)^{k(p-k)}\star'_g (\pi_{2,g,\phi}^*)_{r(\gamma)} \star) \xi \\
&=& <\eta, (-1)^{k(p-k)}\star'_g(\pi_{2,g,\phi}^*)_{r(\gamma)} \star) \xi>(\gamma)
\end{eqnarray*}
where in the above computation we have used the fact that $\pi_{2,g,\phi}^*$ is the inverse of $\lambda_{g,\phi}^*$ on cohomology and since $\lambda_g$ is a homotopy equivalence, $\lambda^*_{g,\phi}$ preserves fundamental cycles. Therefore, $(\lambda_{g,\phi}^*)^\sharp_{x}= (-1)^{k(p-k)}\star'_g(\pi_{2,g,\phi}^*)_{x} \star $, and so it induces an adjoint on the Hilbert-modules given by $(\lambda_{g,\phi}^*)^\sharp= (-1)^{k(p-k)}T'_g \pi_{2,g,\phi}^* T_X = (-1)^{k(p-k)}T'_g \Psi^\phi_g T_X $. Thus we have,
\begin{eqnarray*}
\label{lhsf}
&&<\psi_1\otimes \psi_2, \epsilon_f \Psi^\phi_f \rho^1_{h'} (Id\otimes (\Xi^X_{f,g})^{-1})(\eta_1\otimes \eta_2\otimes\eta_3)>\\
\hspace{-2cm}&=& <\epsilon_g\circ ((-1)^{k(p-k)}T'_g \Psi^\phi_g T_X)\psi_1\otimes \psi_2, \eta_1\otimes \eta_2\otimes\eta_3>\nonumber\\
\hspace{-2cm}&=& (-1)^{k(p-k)} <(\epsilon_g\circ T'_g \otimes Id)(\Psi^\phi_g\otimes Id) (T_X\psi_1\otimes \psi_2), \eta_1\otimes \eta_2\otimes\eta_3>
\end{eqnarray*}
Now a similar computation gives,
\begin{eqnarray}
\label{rhsf}
&&< (-1)^{k(p-k)}[T'_{X'} \Lambda' (g^*_\phi \otimes Id) (T_X\otimes Id)](\psi_1\otimes \psi_2),\Lambda'(\eta_1\otimes \eta_2\otimes \eta_3)>\\
\hspace{-3cm}&=& < [\Lambda' (T'_g\otimes Id) (\epsilon_g\Psi^\phi_g \otimes Id) (T_X\otimes Id)](\psi_1\otimes \psi_2),\Lambda'(\eta_1\otimes \eta_2\otimes \eta_3)>\nonumber\\
\hspace{-3cm}&=& (-1)^{k(p-k)}< (T'_g\otimes Id) (\epsilon_g\Psi^\phi_g \otimes Id) (T_X\otimes Id)(\psi_1\otimes \psi_2),\eta_1\otimes \eta_2\otimes \eta_3>\nonumber\\
\hspace{-3cm}&=& (-1)^{k(p-k)} < (\epsilon_g T'_g\otimes Id) (\Psi^\phi_g \otimes Id) (T_X\otimes Id)(\psi_1\otimes \psi_2),\eta_1\otimes \eta_2\otimes \eta_3>,
\end{eqnarray}
where in the above computations we have used the facts that $\Lambda'$ is an isometric isomorphism and $\epsilon_g$ intertwines the Poincar\'{e} duality operators. The above computations thus yield the equality on closed $k$-forms,
$$
 <\psi_1\otimes \psi_2, f^*_\phi\Lambda'(\eta_1\otimes\eta_2\otimes\eta_3)>_{\maE_{X,E}\otimes \maE^X_{X'}(f)} = (-1)^{k(p-k)}< [T'_{X'} \Lambda' (g^*_\phi \otimes Id) (T_X\otimes Id)](\psi_1\otimes \psi_2),\Lambda'(\eta_1\otimes \eta_2\otimes \eta_3)>_{\maE_{X',E'}}.
 $$
 Thus the proof of the proposition is complete.
\end{proof}

\end{document}